\documentclass{article}

\newif\ifarxiv
\arxivfalse % for arxiv % if activated need to delete .bbl file
\arxivtrue % for peer reviewed 

\ifarxiv % for arxiv
	\usepackage[final]{hyperref}
	\usepackage{geometry}
	\usepackage[
	backend=bibtex,
	style=alphabetic,
	sorting=ynt,
	maxnames=100,
	maxcitenames=3,
	]{biblatex}
	\usepackage{color}

	\definecolor{mydarkblue}{rgb}{0,0.08,0.45}
	\definecolor{mycitationcolor}{rgb}{0,0.08,0.45}
	\definecolor{mylinkcolor}{rgb}{1,0.0,0} %red
	\definecolor{mylinkcolor}{rgb}{0,0.08,0.85}%{0.58, 0.15, 0.14}
	\hypersetup{ %
		pdfborder=0 0 0,
		colorlinks=true,
		linkcolor=mylinkcolor,
		citecolor=mycitationcolor,
		filecolor=mydarkblue,
		urlcolor=mydarkblue,
	}
\else

	\usepackage{hyperref}       % hyperlinks
	\usepackage{color}

\fi

\usepackage{amsmath, amssymb, amsfonts, amsthm}
\usepackage[utf8]{inputenc} % allow utf-8 input
\usepackage[T1]{fontenc}    % use 8-bit T1 fonts
\usepackage{lmodern}
\usepackage{url}            % simple URL typesetting
\usepackage{booktabs}       % professional-quality tables
\usepackage{nicefrac}       % compact symbols for 1/2, etc.
\usepackage{microtype,xspace,bm}      % microtypography
\usepackage{multirow}

\usepackage[linesnumbered,ruled,vlined]{algorithm2e}
\SetKwInput{Init}{Initialize}
\SetKwComment{Comment}{$\triangleright$\ }{}

\newcommand{\defeq}{\mathrel{\mathop:}=}

\newcommand{\argmin}{\mathop{\mathrm{argmin}}}
\newcommand{\argmax}{\mathop{\mathrm{argmax}}}

\newcommand{\iprod}[2]{\langle #1, #2 \rangle}

\newcommand{\norm}[1]{\|{#1} \|}

\newcommand{\xhat}{{\widehat{x}}}
\newcommand{\yhat}{{\widehat{y}}}

\newcommand{\cX}{\mathcal{X}}
\newcommand{\cY}{\mathcal{Y}}
\newcommand{\PY}[1]{\mathcal{P}_{\cY}\left({#1}\right)}

\newcommand{\order}[1]{O\left(#1\right)}
\newcommand{\otilde}[1]{\widetilde{O}\left(#1\right)}

\newcommand{\papertitle}{Efficient Algorithms for Smooth Minimax Optimization}

\title{\papertitle}

\ifarxiv
\else
\usepackage{neurips_2019}

\author{%
	Kiran K.~Thekumprampil \\
	University of Illinois at Urbana-Champaign\\
	\texttt{thekump2@illinois.edu} \\
	\And
	Prateek Jain\\
	Microsoft Research, India\\
	\texttt{prajain@microsoft.com} \\
	\And
	Praneeth Netrapalli\\
	Microsoft Research, India\\
	\texttt{praneeth@microsoft.com}\\
	\And
	Sewoong Oh\\
	University of Washington, Seattle\\
	\texttt{sewoong@cs.washington.edu}\\
}
\fi

\usepackage{mathtools}

\def\reals{\mathbb{R}}

\def\cX{\mathcal{X}}

\def\simplex{\Delta}

\def\tx{\tilde{x}}
\def\ty{\tilde{y}}

\def\tepsilon{\tilde{\varepsilon}}

\def\wf{\widehat{f}}
\def\wg{\widehat{g}}

\newcommand{\Ip}[2]{\left\langle#1, #2\right\rangle}

\DeclarePairedDelimiter{\ceil}{\lceil}{\rceil}

\newtheorem{theorem}{Theorem}
\newtheorem{lemma}{Lemma}
\newtheorem{definition}{Definition}
\newtheorem{corollary}{Corollary}

\newcommand{\mpscc}{{DIAG}\xspace}
\newcommand{\cmpscc}{{C-DIAG}\xspace}
\newcommand{\mpstep}{\texttt{Imp-STEP}\xspace}
\newcommand{\lqancc}{{Prox-DIAG}\xspace}
\newcommand{\lqafncc}{{\selectfont Prox-FDIAG}\xspace}
\newcommand{\alqafncc}{{\selectfont Adaptive Prox-FDIAG}\xspace}

\usepackage{empheq}
\newcommand*\widefbox[1]{\fbox{\hspace{2em}#1\hspace{2em}}}

\begin{document}
	
\ifarxiv
%%%%%% Custom title
\date{}
%\author{Kiran Koshy Thekumprampil, Prateek Jain, Praneeth Netrapalli,  Sewoong Oh}
\author{
	Kiran Koshy Thekumparampil$^\dagger$, Prateek Jain$^\ddagger$,  Praneeth Netrapalli$^\ddagger$, Sewoong Oh$^\pm$
	\thanks{Author emails are {\text thekump2@illinois.edu}, \text{prajain@microsoft.com}, {\text{praneeth@microsoft.com}, and \text{sewoong@cs.washington.edu}}. 
		%This work used the Extreme Science and Engineering Discovery Environment (XSEDE), which is supported by National Science Foundation grant number OCI-1053575.  Specifically, it used the Bridges system, which is supported by NSF award number ACI-1445606, at the Pittsburgh Supercomputing Center (PSC).
	}\\
	\\
	$^\dagger$University of Illinois at Urbana-Champaign,
	$^\ddagger$Microsoft Research, India,\\
	$^\pm$University of Washington, Seattle
}
\else
\fi

\maketitle

\begin{abstract}
This paper studies first order methods for solving smooth minimax optimization problems $\min_x \max_y g(x,y)$ where $g(\cdot,\cdot)$ is smooth and $g(x,\cdot)$ is concave for each $x$. In terms of $g(\cdot,y)$, we consider two settings -- strongly convex and nonconvex -- and improve upon the best known rates in both. For strongly-convex $g(\cdot, y),\ \forall y$, we propose a new algorithm combining Mirror-Prox and Nesterov's AGD, and show that it can find global optimum in $\otilde{1/k^2}$ iterations, improving over current state-of-the-art rate of $O(1/k)$. We use this result along with an inexact proximal point method to provide $\otilde{1/k^{1/3}}$ rate for finding stationary points in the nonconvex setting where $g(\cdot, y)$ can be nonconvex. This improves over current best-known rate of $O(1/k^{1/5})$. Finally, we instantiate our result for finite nonconvex minimax problems, i.e., $\min_x \max_{1\leq i\leq m} f_i(x)$, with nonconvex $f_i(\cdot)$, to obtain convergence rate of $O(m(\log m)^{3/2}/k^{1/3})$ total gradient evaluations for finding a stationary point.% which improves the current best known rates %but with significantly smaller dependence on $m$. 
\end{abstract}

% !TEX root = ./smooth_minimax.tex
\section{Introduction}
\label{sec:intro}
In this paper we study smooth minimax problems of the form: 
\begin{align}\label{eqn:minimax}
	\min_{x \in \mathcal{X}} \; \max_{y\in \mathcal{Y}} \;g(x,y)\;\;,\ \ g: {\mathcal{X}\times \mathcal{Y}}\rightarrow \mathbb{R}, \; g \mbox{ is smooth i.e., gradient Lipschitz}.
\end{align}
%where $f(x,\cdot)$ is concave for every $x$
The problem has applications in several domains such as machine learning~\cite{goodfellow2014generative,madry2017towards}, optimization~\cite{bertsekas2014constrained}, statistics~\cite{berger2013statistical}, mathematics~\cite{kinderlehrer1980introduction}, and game theory~\cite{myerson2013game}. Given the importance of these problems, there is an extensive body of work that studies various algorithms and their convergence properties. The vast majority of existing results for this problem focus on the convex-concave setting, where $g(\cdot,y)$ is convex for every $y$ and $g(x,\cdot)$ is concave for every $x$. 
%Most of the existing results focus on convex-concave minimax problems where $g(\cdot, y)$ is convex.
The best known convergence rate in this setting is $O(1/k)$ for the primal-dual gap, achieved for example by Mirror-Prox~\cite{nemirovski2004prox}. 
%proposed the Mirror-Prox algorithm which gives a convergence rate of . 
%\textcolor{blue}{
This rate is also known to be optimal for the class of smooth convex-concave problems \cite{ouyang2018lower}. 
A natural question is 
whether we can achieve a faster convergence if we have strong convexity (as opposed to just convexity) of $g(\cdot,y)$. 
We answer this in the affirmative, 
by introducing an algorithm that achieves a convergence rate of 
$\otilde{1/k^2}$ for the general smooth, strongly-convex--concave minimax problem. 
The algorithm we propose is a novel combination of Mirror-Prox and Nesterov's accelerated gradient descent. 
This matches the known lower bound 
of $\Omega(1/k^2)$ from \cite{ouyang2018lower}, 
closing the gap up to a poly-logarithmic factor.  
The only known upper bounds  
that obtain a rate of $O(1/k^2)$ in this context 
are for very special cases,  
where $x$ and $y$ are connected through a bi-linear term 
or $g(x,\cdot)$ is linear in $y$
\cite{nesterov2005excessive,juditsky2011first,GOS14,chambolle2016ergodic,he2016accelerated,xu2017iteration,hamedani2018primal,xie2019accelerated}. 
%}
%
%Our first contribution is to close this gap by 
%presenting an algorithm that achieves convergence rate of 
%$\otilde{1/k^2}$ for the general smooth, strongly-convex--concave minimax problem. 
%The algorithm we propose is a novel combination of Mirror-Prox and Nesterov's accelerated gradient descent.
%
%This matches the known lower bound 
%of $\Omega(1/k^2)$ from \cite{ouyang2018lower} 
%in terms of its dependence in $k$ up to a poly-logarithmic factor, 
%settling the question of  

While most theoretical results focus on the convex-concave setting, several real world problems fall outside this class. A slightly larger class, which captures several more applications, is the class of smooth nonconvex--concave minimax problems, where 
%\praneeth{TODO}
%However, several real world problems naturally reduce to nonconvex concave minimax problems, i.e, problems where 
$g(x, \cdot)$ is {\em concave} for every $x$ but $g(\cdot, y)$ can be nonconvex. For example, finite minimax problems, i.e., $\min_x \max_{i=1}^m f_i(x) = \min_x \max_{0\preceq y\preceq 1, \sum_{i=1}^m y_i=1} \sum_i y_i \cdot f_i(x)\defeq g(x,y)$ belong to this class,
%where $g(x, \cdot)$ is a linear function. 
and so do nonconvex constrained optimization problems~\cite{KomiyamaTHS18}. In addition, several machine learning problems with non-decomposable loss functions \cite{kar2015surrogate} also belong to this class. 

%Despite widespread usage of nonconvex concave minimax problems, so far, there are limited results in this domain.

%While, no algorithm with a better convergence rate is known even if we further assume that $f(\cdot,y)$ is strongly convex $\forall y\in \mathcal{Y}$, it is not clear if $O(1/k)$ is the best achievable rate for this setting as well. In fact, for a special case of such strongly-convex--concave problems (with only a linear coupling between $x$ and $y$),~\cite{nesterov2005excessive}  gives a method with convergence rate $O(1/k^2)$.
%However, the rate is not optimal for strongly-convex concave problems. Using a novel view of Mirror-Prox algorithm and combining it with the Nesterov's accelerated gradient descent (AGD) method, we obtain a faster rate (as well as computational complexity) of $O(1/k^2)$ for {\em general strongly-convex concave}  minimax problems. 

In this general nonconvex concave setting however, we cannot hope to find global optimum efficiently as even the special case of nonconvex optimization is NP-hard. Similar to nonconvex optimization, we might hope to find an approximate stationary point \cite{nesterov1998introductory}.
%In contrast to nonconvex optimization, the definition of stationarity is not obvious in this setting. For example, ~\cite{daskalakis2018limit} and others consider the notion of local Nash equilibrium, but such points need-not even exist for some problems.~\cite{jin2019minmax} argues that this notion is not meaningful in the context of machine learning applications and proposes a new notion of optimality that is also used by \cite{rafique2018non} and reduces to the Moreau-envelop based notion of stationarity for the general non-smooth non-convex optimation; we also study this notion. 
%Independent of this work, \cite{nouiehed2019solving} propose another notion of approximate stationarity but it can be shown \cite{nouiehed2019solving}'s notion is implied by the notion by \cite{rafique2018non,jin2019minmax}.  
%\cite{jin2019minmax} shows that a simple gradient based algorithm converges to a first order stationary point in $O(1/k^{1/5})$ iterations. A few recent results have shown improved rates for special cases of the problem. \cite{davis2018stochastic} proposed a proximal SGD based method that converges to the stationary point of weakly-convex functions (see Definition~\ref{def:weak-convex}) at $O(1/k^{1/4})$.
%\cite{rafique2018non} proposes a proximal method with convergence rate of 
%$O(1/k^{1/6})$ for finite minimax over $m$ functions and 
%$O(1/k^{1/6})$ for smooth minimax problems which is incomparable to the $O({1}/{k^{1/3.5}})$ rate obtained by \cite{nouiehed2019solving} due to a   weaker notion of stationarity. 

Our second contribution is a new algorithm and a faster rate for the {\em general smooth nonconvex--concave} minimax problem. Our algorithm 
%iteratively constructs and optimizes local quadratic approximations for the 
is an inexact proximal point method for the nonconvex function $f(x) \defeq \max_{y\in \mathcal{Y}} g(x,y)$. The key insight is that the proximal point problem
%local quadratic approximation 
in each iteration results in a strongly-convex concave minimax problem, for which we use our improved algorithm to obtain the overall computation/iteration complexity of $\otilde{1/k^{1/3}}$ thus improving over the previous best known rate of $O(1/k^{1/5})$~\cite{jin2019minmax}\footnote{While~\cite{jin2019minmax} gives a rate of $\order{1/k^{1/4}}$ with an approximate maximization oracle for $\max_{y\in \mathcal{Y}} g(x,y)$, taking into account the cost of implementing such a maximization oracle gives a rate of $\order{1/k^{1/5}}$.}. % for the same notion of approximate stationarity. %As mentioned above, the independently proposed rate of $O(1/k^{3.5})$ by \cite{nouiehed2019solving} is incomparable as it uses a weaker notion of stationarity. 
%Our result directly implies a convergence rate of $O(1/k^{1/3})$ for a special class of nonsmooth nonconvex minimization problems (i.e., those that can be written as smooth nonconvex concave minimax problems) improving upon the best known rate 
%significantly upon the result of \cite{jin2019minmax} and is in fact better than some of the specialized results that address a sub-class of nonconvex concave problems such as weakly convex functions \cite{davis2018stochastic}. 

Finally, we specialize our result to finite minimax problems, i.e., $\min_x \max_{1\leq i\leq m} f_i(x)$ where $f_i(x)$ can be nonconvex function but each $f_i$ is a smooth function; nonconvex constrained optimization problems can be reduced to such finite minimax problems. For these, we obtain a rate of $\otilde{m(\log m)^{3/2}/k^{1/3}}$ total gradient computations which improves upon the state-of-the-art rate ($O(m\sqrt{\log m}/k^{1/5})$) in this setting as well.

\ifarxiv
\bigskip \noindent
\fi
\textbf{Summary of contributions}: See also Table~\ref{tab:results}.\\
1. $\otilde{1/k^2}$ convergence rate for smooth, strongly-convex -- concave problems, improving upon the previous best known rate of $\order{1/k}$ and, \\
2. $\otilde{1/k^{1/3}}$ convergence rate for smooth, nonconvex -- concave problems, improving upon the previous best known rate of $\order{1/k^{1/5}}$.
% for : a) strongly-convex concave problems and b) nonconvex concave problems, 
%c) nonconvex finite minimax problems, 
%and compares them to the respective existing state-of-the-art results. 
%\begin{table}[t]
%%	\renewcommand{\arraystretch}{1.3}
%	\begin{center}
%		\label{tab:results}
%		\begin{tabular}{|c|c|c|c|}
%			\hline
%			\textbf{Setting} & \textbf{Optimality notion} & \textbf{\begin{tabular}{@{}c@{}} Previous \\ state-of-the-art \end{tabular}} & \textbf{Our results} \\
%			\hline
%			Convex & Primal-dual gap & \begin{tabular}{@{}c@{}} $\order{k^{-1}}$~\cite{nemirovski2004prox} \end{tabular} & - \\
%			\hline
%			Strongly convex & Primal-dual gap & \begin{tabular}{@{}c@{}} $\order{k^{-1}}$~\cite{nemirovski2004prox} \end{tabular} & $\otilde{{k^{-2}}}$ \\
%			\hline
%			Nonconvex & \begin{tabular}{@{}c@{}} Approximate \\ stationary point \end{tabular} & $\order{k^{-1/5}}$~\cite{jin2019minmax} & $\otilde{{k^{-1/3}}}$ \\
%			\hline
%		\end{tabular}
%	\end{center}
%	\caption{Comparison of our results with previous state of the art. We assume that $g(\cdot,\cdot)$ is smooth (i.e., has Lipschitz gradients) and $g(x,\cdot)$ is concave $\forall x \in \mathcal{X}$. Convexity, strong convexity and nonconvexity in the first column refers to $g(\cdot,y)$ for fixed $y$.}
%%	\renewcommand{\arraystretch}{1.}
%\end{table}
\begin{table}[t]
	\renewcommand{\arraystretch}{1.3}
	\begin{center}
		\label{tab:results}
		\begin{tabular}{c c c c c}
			\toprule
			\textbf{Setting} & \textbf{Optimality notion} & \textbf{\begin{tabular}{@{}c@{}} Previous \\ state-of-the-art \end{tabular}} & \textbf{Our results} & \textbf{Lower bound}\\
			\midrule
			Convex & Primal-dual gap & \begin{tabular}{@{}c@{}} $\order{k^{-1}}$~\cite{nemirovski2004prox} \end{tabular} & - & $\Omega(k^{-1})$ \cite{ouyang2018lower}\\
			
			Strongly convex & Primal-dual gap & \begin{tabular}{@{}c@{}} $\order{k^{-1}}$~\cite{nemirovski2004prox} \end{tabular} & $\otilde{{k^{-2}}}$ & $\Omega(k^{-2})$ \cite{ouyang2018lower} \\
			
			Nonconvex & \begin{tabular}{@{}c@{}} Approx.~stat.~point \end{tabular} & $\order{k^{-1/5}}$~\cite{jin2019minmax} & $\otilde{{k^{-1/3}}}$ & -\\
			\bottomrule
		\end{tabular}
	\end{center}
	\caption{Comparison of our results with previous state-of-the-art. We assume that $g(\cdot,\cdot)$ is smooth (i.e., has Lipschitz gradients) and $g(x,\cdot)$ is concave $\forall x \in \mathcal{X}$. Convexity, strong convexity and nonconvexity in the first column refers to $g(\cdot,y)$ for fixed $y$.}
\end{table}

\ifarxiv
\bigskip \noindent
\fi
% !TEX root = smooth_minimax.tex
{\bf Related works}: %\label{sec:related}
%{\color{blue} 
For strongly-convex-concave minimax problems with special structures, 
several algorithms have been proposed. 
In an increasing order of generality, 
\cite{GOS14,xu2017iteration,xu2018accelerated} 
study 
optimizing a strongly convex function with  linear constraints, 
which can be posed as a special case of minimax optimization, 
\cite{nesterov2005excessive} studies 
a minimax problem where $x$ and $y$ are connected only through a bi-linear term, 
and \cite{hamedani2018primal} and \cite{juditsky2011first} study a case where 
$g(x, \cdot )$ is linear in $y$.   
In all these cases, it is shown that 
$O(1/k^2)$ convergence rate is achievable if $g(\cdot,y)$ is strongly-convex $\forall \; y$. 
Recently, 
\cite{zhao2019optimal} provides a unified approach, that achieves $O(1/k)$ convergence rate for 
general convex-concave case and $O(1/k^2)$ 
for a special case with strongly-convex $g(\cdot,y)$ 
and linear $g(x,\cdot)$. 
%Separates out complexities for $L_{\rm xy}$, $L_{\rm xx}$, and $L_{\rm yy}$. Also studies stochastic setting.
However,  
it has remained an open question if the fast rate of $O(1/k^2)$ can be achieved for general strongly-convex-concave minimax problems. 
% which is critical in  achieving the desired fast rate in the general  {\em non-convex}  concave minimax problems that we are interested in.
%}

For nonconvex-concave minimax problems, \cite{rafique2018non} considers both deterministic and stochastic settings, and proposes inexact proximal point methods for solving smooth nonconvex--concave problems. In the deterministic setting, their result guarantees an error of $O(1/k^{1/6})$.
We note that there have also been other notions of stationarity proposed in literature for nonconvex-concave minimax problems~\cite{lu2019hybrid,nouiehed2019solving}. 
These notions however are weaker than the one considered in this paper, in the sense that, our notion of stationarity implies these other notions (without loss in parameters). For one such weaker notion,~\cite{nouiehed2019solving} proposes an algorithm with a convergence rate of $\order{1/k^{3.5}}$. Since the notion they consider is weaker, it does not imply the same convergence rate in our setting.

%\begin{itemize}
%	\item Variational inequalities
%	\item Saddle point problems
%	\item Nonconvex optimization -- smooth and nonsmooth
%	\item Stochastic methods
%\end{itemize}
%
%\textbf{Minimax optimization}: Starting from the seminal work of~\cite{neumann1928theorie}, minimax optimization has been widely studied and applied in several fields. While initial results on designing algorithms focused on the bilinear setting~\cite{robinson1951iterative}, there has been substantial amount of work generalizing it to the convex-concave setting~\cite{korpelevich1976extragradient,nemirovsky1978}. In the context of optimization, these problems arise often in constrained convex optimization~\cite{bertsekas2014constrained} and have been widely studied in that context as well.
%\textbf{Variational inequalities}:
We would also like to highlight the work on {\em variational inequalities} that are a generalization of minimax optimization problems. In particular,  monotone variational inequalities  generalizes the convex-concave minimax problems  and have applications in  solving differential equations~\cite{kinderlehrer1980introduction}. There have also been a large number of works designing efficient algorithms for finding solutions to monotone variational inequalities~\cite{bruck1977weak,nemirovsky1981,nemirovski2004prox}.

\ifarxiv
\bigskip \noindent
\fi
{\bf Notations}:  $\reals$ is the real line and for any natural number $p$, $\reals^p$ is the real vector space of dimension $p$. $\|\cdot\|$ is a norm on some metric space which would be evident from the context. For a convex set $\cX \subseteq \reals^p$ and $x \in \reals^p$, $\mathcal{P}_{\cX}(x) = \arg\min_{x' \in \cX} \|x-x'\|$ is the projection of $x$ on to $\cX$. For a differentiable function $g(x, y)$, $\nabla_x g(x, y)$ is its gradient with respect to $x$ at $(x, y)$. We use the standard big-O notations. 
For functions $T, S: \reals \to \reals$ such that  $0 < \lim\inf_{x \to \infty} T(x), \lim\inf_{x \to \infty} S(x)$, (a)  $T(x) = O(S(x))$ means $\lim\sup_{x \to \infty} T(x)/S(x) < \infty$; (b) $T(x) = \Theta(S(x))$ means $T(x) = O(S(x))$ and $S(x) = O(T(x))$; and (c) $T(x) = \otilde{S(x)}$ means that $T(x) = O(S(x)R(x))$ for some poly-logarithmic function $R: \reals \to \reals$.

\ifarxiv
\bigskip \noindent
\fi
\textbf{Paper organization}: In Section~\ref{sec:background}, we present preliminaries and all relevant background. In Section~\ref{sec:str-cvx}, we present our results for strongly-convex--concave setting and in section~\ref{sec:noncvx}, results for nonconvex--concave setting. In Section~\ref{sec:experiments}, we present empirical evaluation of our algorithm for nonconvex-concave setting and compare it to a state-of-the-art algorithm.  We conclude in Section~\ref{sec:conc}. Several technical details are presented in the appendix.
%\input{related.tex}
% !TEX root = smooth_minimax.tex
\section{Preliminaries and background material}\label{sec:background}
In this section, we will present some preliminaries, describing the setup and reviewing some background material that will be useful in the sequel.
\subsection{Minimax problems}
We are interested in the minimax problems of the form \eqref{eqn:minimax} where $g(x,y)$ is a smooth function. %solving problems of the form: 
%\begin{align}
%	\min_{x \in \mathcal{X}} \max_{y\in \mathcal{Y}} g(x,y), \label{eqn:minimax}
%\end{align}
%where $g(x,y)$ is a smooth function.
%and $\mathcal{Y}$ is a compact convex set (we do not need compactness of $\mathcal{X}$ for reasons that will become clear later).
\begin{definition}\label{def:smooth}
	A function $g(x,y)$ is said to be $L$-smooth if: 
	\begin{equation*}
	\max\left\{\norm{\nabla_x g(x,y) - \nabla_x g(x',y')}, \norm{\nabla_y g(x,y) - \nabla_y g(x',y')} \right\}  \leq L \left(\norm{x - x'} + \norm{y - y'}\right). 
	\end{equation*}
\end{definition}
Throughout, we assume that $g(x,.)$ is {\em concave} for every $x \in \mathcal{X}$. For $g(\cdot, y)$ behavior in terms of $x$, there are broadly two settings: 
\subsubsection{Convex-concave setting}
In this setting, $g(\cdot,y)$ is convex $\forall \; y \in \mathcal{Y}$. Given any $g$ and $\forall (\xhat,\yhat)$, the following holds trivially: 
\begin{align*}
	\min_{x \in \mathcal{X}} g(x,\yhat) \leq g(\xhat,\yhat) \leq \max_{y\in \mathcal{Y}} g(\xhat, y),
\end{align*}
which then implies that $\max_{y\in \mathcal{Y}} \min_{x \in \mathcal{X}} g(x,y) \leq \min_{x \in \mathcal{X}} \max_{y\in \mathcal{Y}} g(x,y)$. The celebrated minimax theorem for the convex-concave setting \cite{sion1958general} says that if  $\mathcal{Y}$ is a compact set then the above inequality is in fact an equality, i.e., $\max_{y\in \mathcal{Y}} \min_{x \in \mathcal{X}} g(x,y) = \min_{x \in \mathcal{X}} \max_{y\in \mathcal{Y}} g(x,y)$. Furthermore, any point $(x^*,y^*)$ is an optimal solution to~\eqref{eqn:minimax} if and only if: 
\begin{align}
	\min_{x \in \mathcal{X}} g(x,y^*) = g(x^*,y^*) = \max_{y\in \mathcal{Y}} g(x^*, y).
\end{align}
Hence, our goal is to find $\varepsilon$-primal-dual pair $(\xhat,\yhat)$ with small primal-dual gap: $\max_{y\in \mathcal{Y}} g(\xhat,y) - \min_{x \in \cX} g(x,\yhat)$. 
\begin{definition}\label{def:eps_fosp_cvx}
	For a convex-concave function $g: \cX \times \cY \to \reals$ %with zero duality gap: $\max_{y\in \mathcal{Y}} \min_{x \in \mathcal{X}} g(x,y) = \min_{x \in \mathcal{X}} \max_{y\in \mathcal{Y}} g(x,y)$
	, $(\hat{x}, \hat{y})$ is an $\varepsilon$-primal-dual-pair of $g$ if the primal-dual gap is less than $\varepsilon$: $\max_{y\in \mathcal{Y}} g(\xhat,y) - \min_{x \in \cX} g(x,\yhat) \leq \varepsilon$. 
\end{definition}
\subsubsection{Nonconvex-concave setting} \label{sec:nonconvex-concave-minmax}
In this setting the function $g(\cdot,y)$ need not be convex. One cannot hope to solve such problems in general, since the special case of nonconvex optimization is already NP-hard \cite{nouiehed2018convergence}. Furthermore, the minimax theorem no longer holds, i.e., $\max_{y\in \mathcal{Y}} \min_{x \in \mathcal{X}} g(x,y)$ can be strictly smaller than $\min_{x \in \mathcal{X}} \max_{y\in \mathcal{Y}} g(x,y)$.  Oftentimes the order of $min$ and $max$ might be important for a given application i.e., we might be interested only in minimax but not maximin (or vice versa). So, the primal-dual gap may not be a meaningful quantity to measure convergence. One approach, inspired by nonconvex optimization, to measure convergence is to consider the function $f(x)=\max_{y\in \mathcal{Y}} g(x,y)$ and consider the convergence rate to approximate first order stationary points (i.e., $\nabla f(x)$ is small)\cite{rafique2018non,jin2019minmax}. But as $f(x)$ could be non-smooth, $\nabla f(x)$ might not even be defined. It turns out that whenever $g(x,y)$ is smooth, $f(x)$ is weakly convex (Definition~\ref{def:weak-convex}) for which first order stationarity notions are well-studied and are discussed below. 

\bigskip\noindent
\textbf{Approximate first-order stationary point for weakly convex functions}: We first need to generalize the notion of gradient for a non-smooth function.
\begin{definition}\label{def:Frechet}
	The Fr\'echet sub-differential of a function $f(\cdot)$ at $x$ is defined as the set, $\partial f(x) = \{u \,|\,\ \underset{x' \to x}{\lim\inf} {f(x') - f(x) - \Ip{u}{x'-x}}/{\|x' - x \|} \geq 0 \}$.
\end{definition}
%Subgradient methods applied to non-smooth non-convex function $f(x)$ 
%aims to find an approximate stationary point defined with respect to its Moreau envelope.
In order to define approximate stationary points, we also need the notion of weakly convex function and Moreau envelope.
\begin{definition}\label{def:weak-convex}
	A function $f:\cX \to \reals \cup \{\infty\}$ is {\em $L$-weakly convex} if,
	\begin{align}
	f(x) + \Ip{u_x}{x'-x} - \frac{L}{2} \|x' - x\|^2  \;\; \leq \;\; f(x')\,, 
	\label{eq:weakly-cvx}
	\end{align}
	for all Fr\'echet subgradients $u_x \in \partial f(x)$.
\end{definition}
\begin{definition}\label{def:Moreau}
For a proper lower semi-continuous (l.s.c.) function $f:\cX \to \reals \cup \{\infty\}$ and $\lambda > 0$ ($\cX \subseteq \reals^p$), the {\em Moreau envelope} function is given by 
\begin{eqnarray}
f_{\lambda}(x) \;\; = \;\; \min_{x' \in \cX} f(x') + \frac1{2\lambda} \|x - x'\|^2\;. 
\label{eq:envelope}
\end{eqnarray}
\end{definition}
The following lemma provides some useful properties of the Moreau envelope for weakly convex functions. The proof can be found in Appendix~\ref{sec:proof_moreau-properties}.
\begin{lemma}\label{lem:moreau-properties}
	For an $L$-weakly convex proper l.s.c.~function $f:\cX \to \reals \cup \{\infty\}$ ($\cX = \reals^p$) such that $L < 1/\lambda$, the following hold true,
	\begin{enumerate}
		\item[(a)] The minimizer $\hat{x}_{\lambda}(x) = \arg\min_{x' \in \cX}  f(x') + \frac1{2\lambda} \|x - x'\|^2$ is unique and $f(\hat{x}_{\lambda}(x)) \leq f_{\lambda}(x) \leq f(x)$. Furthermore, $\arg\min_x f(x) = \arg\min_x f_{\lambda} (x)$.
		\item[(b)] $f_{\lambda}$ is $\big(\frac1\lambda + \frac1{\lambda(1 - \lambda L)}\big)$-smooth and thus differentiable, and
		\item[(c)] $ \min_{u \in \partial f(\hat{x}_{\lambda}(x))} \|u\| \leq  (1 / \lambda) \| \hat{x}_{\lambda}(x) - x\|  =  \| \nabla f_{\lambda}(x) \|$.
	\end{enumerate}
\end{lemma}
Now, first order stationary point of a non-smooth nonconvex function is well-defined, i.e.,  $x^*$ is a {\em first order stationary point (FOSP)} of a  function $f(x)$ if, $0 \in \partial f(x^*)$ (see Definition \ref{def:Frechet}). However, unlike smooth functions, it is nontrivial to define an \emph{approximate} FOSP. For example, if we define an $\varepsilon$-FOSP as the point $x$ with $\min_{u \in \partial f(x)} \| u \| \leq \varepsilon$, there may never exist such a point for sufficiently small $\varepsilon$, unless $x$ is exactly a FOSP. In contrast, by using above properties of the Moreau envelope of a weakly convex function, it's approximate FOSP can be defined as \cite{davis2018stochastic}: 
%\end{definition}
%\begin{definition}\label{def:eps_fosp}
%	We say that $x^*$ is an $\varepsilon$-first order stationary point ($\varepsilon$-FOFS) of the problem \eqref{prob:minimax} if,  $\min_{g\in \partial f(x^*)} \| g\| \leq \varepsilon$, where $\partial f(x)$ is the sub-differential of the function $f(x)$ at $x$.
%\end{definition}
%\noindent
%This definitions ensures that we recover the standard FOSP condition $0 \in \partial f(x)$ when $\varepsilon = 0$. A recent work showed that the stochastic sub-gradient descent on $l$-weakly convex functions takes $16 (l L (f(x) - f^*))^2/\varepsilon^4$ steps to reach an $\varepsilon$-FOSP \cite{davis2018stochastic}.
%We use the following definition for approximate FOSP which is well-defined for any weakly convex function \cite{davis2018stochastic}.
\begin{definition}\label{def:eps_fosp}
	Given an $L$-weakly convex function $f$, we say that $x^*$ is an $\varepsilon$-first order stationary point ($\varepsilon$-FOSP) if, $\| \nabla f_{\frac{1}{2L}}(x^*) \| \leq \varepsilon$, where $f_{\frac{1}{2L}}$ is the Moreau envelope with parameter $1/2L$.
\end{definition}
\noindent
Using Lemma \ref{lem:moreau-properties}, we can show that for any $\varepsilon$-FOSP $x^*$, there exists $\hat{x}$ such that $\|\hat{x} - x^* \| \leq \varepsilon/2L$ and $\min_{u\in \partial f(\hat{x})} \| u\| \leq \varepsilon$. In other words, an $\varepsilon$-FOSP is $O(\varepsilon)$ close to a point $\hat{x}$ which has a subgradient smaller than $\varepsilon$. We note that other notions of FOSP have also been proposed recently such as in~\cite{nouiehed2019solving}. However, it can be shown that an $\varepsilon$-FOSP according to the above definition is also an $\epsilon$-FOSP with \cite{nouiehed2019solving}'s definition as well, but the reverse is not necessarily true. %This definition recovers the standard FOSP condition $0 \in \partial f(x)$ when $\varepsilon = 0$. A recent work showed that the stochastic sub-gradient descent on $G$-weakly convex functions takes $16 (LG (f(x) - f^*))^2/\varepsilon^4$ steps to reach an $\varepsilon$-FOSP \cite{davis2018stochastic}.
%In the following, we ???
%
%\begin{table}[t]
%	\begin{center}
%		\begin{tabular}{|l|c | c|} 
%			\hline
%			$f(x)$: Assumption & Rate & Reference \\ [0.5ex] 
%			\hline\hline
%			&  &    \\
%			\hline \hline
%			$f(x)$:  weakly-convex  & $1/\varepsilon^4$  & \cite{davis2018stochastic} \\
%			\hline \\
%			\hline
%			$\max_i f_i(x)$: convex and $G$-smooth $f_i$ &    &  Algorithm \ref{algo:minimax_gd} (ours)\\ 
%			\hline
%			$\max_i f_i(x)$: weakly-convex $f_i$  & $LG\sqrt{\log m} /\varepsilon^3$  & Algorithm \ref{algo:minimax_gd} (ours) \\
%			\hline
%			$\max_y g(x, y)$: weakly-convex--concave $g$  & $1/\varepsilon^3$  & Algorithm \ref{algo:minimax_noncvx_cve_gd} (ours) \\
%			\hline
%		\end{tabular}
%	\end{center}
%	\caption{A}
%	\label{tbl:smooth}
%\end{table}

\subsection{Mirror-Prox}
Mirror-Prox~\cite{nemirovski2004prox} is a popular algorithm proposed for solving convex-concave minimax problems~\eqref{eqn:minimax}. It achieves a convergence rate of $\order{1/k}$ for the primal dual gap. The original Mirror-Prox paper~\cite{nemirovski2004prox} motivates the algorithm through a \emph{conceptual} Mirror-Prox (CMP) method, which brings out the main idea behind its convergence rate of $\order{1/k}$. CMP does the following update:
\begin{align}
(x_{k+1},y_{k+1}) = \left(x_{k},y_{k}\right) + \frac{1}{\beta} \left(-\nabla_x g\left(x_{\bm{k+1}},y_{\bm{k+1}}\right), \nabla_y g\left(x_{\bm{k+1}},y_{\bm{k+1}}\right)\right). \label{eqn:cmp}
\end{align}
%Algorithm~\ref{algo:MP} presents pseudocode of CMP.
The main difference between CMP and standard gradient descent ascent (GDA) is that in the $k^{\textrm{th}}$ step, while GDA uses gradients at $(x_k,y_k)$, CMP uses gradients at $(x_{k+1},y_{k+1})$. The key observation of~\cite{nemirovski2004prox} is that if $g(\cdot,\cdot)$ is smooth, it can be implemented efficiently. CMP is analyzed as follows:\\
{\bf Implementability of CMP}: 
%\begin{itemize}
%	\item \textbf{Implementability of CMP}: 
Let $(x_{k}^{(0)},y_{k}^{(0)}) = (x_k,y_k)$. For $\beta < \frac{1}{L}$, the iteration
\begin{align}
(x_{k}^{(i+1)},y_{k}^{(i+1)}) = \left(x_{k},y_{k}\right) + \frac{1}{\beta} \left(-\nabla_x g\left(x_{k}^{(i)},y_{k}^{(i)}\right), \nabla_y g\left(x_{k}^{(i)},y_{k}^{(i)}\right)\right). \label{eqn:fp-iteration}
\end{align}
can be shown to be $\frac{1}{\sqrt{2}}$-contraction (when $g(\cdot,\cdot)$ is smooth) and that its fixed point is $\left(x_{k+1},y_{k+1}\right)$. So, in $\log\frac{1}{\epsilon}$ iterations of~\eqref{eqn:fp-iteration}, we can obtain an accurate version of the update required by CMP. In fact, \cite{nemirovski2004prox} showed that just \emph{two} iterations of~\eqref{eqn:fp-iteration} suffice.\\ %This is the eventual Mirror-Prox algorithm\cite{nemirovski2004prox}.\\
	%\item 
	\textbf{Convergence rate of CMP}: Using CMP update with simple manipulations leads to the following: %for any $x\in \mathcal{X}$ and $y \in \mathcal{Y}$, we have
	{\small 
	\begin{align*}
		g(x_{k+1},y) - g(x,y_{k+1}) \leq \beta\left(\norm{x-x_k}^2 - \norm{x-x_{k+1}}^2 + \norm{y-y_k}^2 - \norm{y-y_{k+1}}^2\right), \forall x\in \mathcal{X},\ y \in \mathcal{Y}.
	\end{align*}}
  $\order{1/k}$ convergence rate follows easily using the above result.
%\end{itemize}
%\begin{algorithm}[t]
%	\DontPrintSemicolon % Some LaTeX compilers require you to use \dontprintsemicolon instead
%	\KwIn{ Smooth function $g(x,y)$, learning rate $\frac{1}{\beta}$, initial point $x_0$}
%	\KwOut{$x_k$}
%	\For{$k = 0, 1, \ldots$} {
%		$\left(x_{k+1},y_{k+1}\right) \leftarrow \left(x_{k},y_{k}\right) + \frac{1}{\beta} \left(-\nabla_x g\left(x_{k+1},y_{k+1}\right), \nabla_y g\left(x_{k+1},y_{k+1}\right)\right)$
%	}
%	\caption{Conceptual Mirror-Prox}
%	\label{algo:MP}
%\end{algorithm}

Finally, our method and analysis also requires Nesterov's accelerated gradient descent method (see Algorithm~\ref{algo:AGD} in Appendix A)and it's per-step analysis by \cite{bansal2017potential} (Lemma~\ref{lem:agd-pf} in Appendix~\ref{sec:nesterov}). 
%\input{moreau.tex}
% ====================================================================================

% ============================================================================================================
% !TEX root = ./smooth_minimax.tex
\section{Strongly-convex concave saddle point problem}\label{sec:str-cvx}
We first study  the minimax problem of the form: 
\begin{align}
\min_{x \in \cX} \; [\; f(x) =  \max_{y \in \cY} g(x, y) \;]\;, 
\tag{P1}\label{prob:cvx-smooth-minimax}
\end{align}
where $g(x, \cdot)$ is concave, $g(\cdot, y)$ is $\sigma$-{\em strongly-convex}, $g(\cdot, \cdot)$ is $L$-smooth, i.e., $0 < \sigma \leq L$.  $\cX = \reals^p$ and $\cY \subset \reals^q$ is a convex compact sub-set of $\reals^q$ and let the function $f$ take a minimum value $f^*$($> -\infty$). Let $D_\cY=\max_{y,y' \in \cY} \|y- y'\|$ be the diameter of $\cY$. 

Our objective here is to find an $\epsilon$-primal-dual pair $(\xhat,\yhat)$ (see Definition~\ref{def:eps_fosp_cvx}). Now the fact that $f(\hat{x}) - f^* \leq \max_{y\in \mathcal{Y}} g(\xhat,y) - \min_{x \in \cX} g(x,\yhat)$ implies that if $(\hat{x}, \hat{y})$ is an $\varepsilon$-primal-dual-pair, then $\hat{x}$ is also an $\varepsilon$-approximate minima of $f$. 
% defined as follows. 
%which is more natural and well-defined for convex functions.  
%by the following more natural definition \ref{def:eps_fosp_cvx}.
%\begin{definition}\label{def:eps_fosp_cvx}
%	For a convex function $f$, we say that $x^*$ is an $\varepsilon$-(convex) first order stationary point ($\varepsilon$-CVX-FOSP) of $f$ if, $f(x^*) - f^* \leq \varepsilon$. 
%\end{definition}
% If $x^*$ is an $\varepsilon$-CVX-FOSP then $x^*$ is a $\sqrt{\ell \varepsilon}$-FOSP (Definition \ref{def:eps_fosp}) of this problem for any $\ell > 0$ (as $f$ is 0-weakly convex), since $f(x) - f^* \geq  f_{\frac1{\ell}}(x) - f^*  \geq \frac1{\ell\\} \|\nabla f_{\frac1{\ell}} (x)\|^2$ (Lemma \ref{lem:moreau-properties} (a, b)). 
Furthermore, by Sion's minimax theorem \cite{komiya1988elementary}, strong-convexity--concavity of $g(\cdot, \cdot)$ ensures that: 
$\min_x [f(x) \defeq \max_y g(x, y)] = \max_y [h(y) \defeq \min_x g(x, y)]$. Hence, one approach to efficiently solving the problem is by optimizing the dual problem $\max_y h(y)$.  By Lemma \ref{lem:outer-max-smooth}, $h(y)$ is an $L(1 + L/\sigma)$-smooth function. So we can use AGD to ensure that $h(y_k)-h(y^*) = O(1/k^2)$. Now, each step of AGD requires computing $\arg\min_x g(x, y_k)$ which can be done efficiently (i.e., logarithmic number of steps)
%achieved in $\log(1/\varepsilon)$ steps using standard gradient descent, 
as $g(\cdot, y_k)$ is strongly-convex and smooth. So, the overall first-order oracle complexity is $h(y_k)-h(y^*) = \otilde{1/k^2}$.
% is $O(\log(1/\varepsilon)/\sqrt{\varepsilon})$. 
\begin{lemma} \label{lem:outer-max-smooth}
	For a $\sigma$-strongly-convex--concave $L$-smooth function $g(\cdot,\cdot)$, 
	$h(u) = \min_{x \in \cX} g(x, u)$ is an $L\big(1+ \frac{L}\sigma \big)$-smooth concave function.
\end{lemma}

So does this simple approach give us our desired result? Unfortunately that is not the case, as the above bound on the dual function $h$ does not translate to the same error rate for primal function $f$, i.e., the solution need not be $\otilde{1/k^2}$-primal-dual pair. E.g., consider $ \min_{x \in \reals } \max_{y \in [-1, 1]} [g(x, y) = xy + x^2/2]$, where $\min_{x} \max_{y} g(x, y) = 0$, $f(x) = x^2/2 + |x|$ and $h(y) = -y^2/2$. If $h(y_k) = \Theta(k^{-2})$, then $x_k \in \argmin_x g(x, y_k) = \Theta(1/k)$ and so $f(x_k)$ is $\Theta(k^{-1})$. 

%A straight-forward but flawed procedure to solve this problem  
%is to solve the dual problem. 
%The inner strongly-convex minimization can be solved in linear $\log(1/\varepsilon)$ time and the outer $L(1 + L/\sigma)$-smooth (Lemma \ref{lem:outer-max-smooth}) concave maximization can get an accelerated convergence of $h(y_k) = O(k^{-2})$ after $k$ iterations. 
 %However, this guarantee on the dual function $h$ does not translate to the same order of convergence for the primal function $f$. As a counter example, 

%Since $g$ is a strongly-convex--concave function, by the following lemma $h(u)$ is smooth.

Instead of using AGD, we introduce a new method to solve the dual problem that we refer to as \mpscc, which stands for Dual Implicit Accelerated Gradient. \mpscc combines ideas from AGD \cite{nesterov1983method} and 
Nemirovski's original derivation of the Mirror-Prox algorithm \cite{nemirovski2004prox}, and can ensure a fast convergence rate of $\tilde{O}(k^{-2})$ for the  primal-dual gap. For better exposition, we first present a conceptual version of DIAG (\cmpscc), which is not implementable \emph{exactly}, but brings out the main new ideas in our algorithm. We then present a detailed error analysis for the \emph{inexact} version of this algorithm, which is implementable.
\subsection{Conceptual version: \cmpscc}
The pseudocode for \cmpscc algorithm is presented in Algorithm~\ref{algo:conceptual_minimax_stronglycvx_cve_agd}. The main idea of the algorithm is in Step 4, where we simultaneously find $x_{k+1}$ and $y_{k+1}$ satisfying the following requirements:
\begin{itemize}
	\item $x_{k+1}$ is the minimizer of $g(\cdot,y_{k+1})$, and
	\item $y_{k+1}$ corresponds to an AGD step (see Algorithm~\ref{algo:AGD} in Appendix~\ref{sec:nesterov}) for $g(x_{k+1},\cdot)$
\end{itemize}
\begin{algorithm}[t]
	\DontPrintSemicolon % Some LaTeX compilers require you to use \dontprintsemicolon instead
	\SetKwFunction{FMPSCC}{MP-SCC}
	\SetKwFunction{FMPSTEP}{Imp-STEP}
	\KwIn{$g$, $L$, $\sigma$, $x_0$, $y_0$, $K$}
	\KwOut{$\bar{x}_{K}, y_{K}$}
	\SetKwProg{Fn}{}{:}{}
	%	\Fn{\FMPSCC{$g$, $L$, $\sigma$, $x_0$, $y_0$, $K$, $\{\varepsilon_{\rm step}^{(k)}\}_{k=1}^K$}}
	{
		Set $\beta \leftarrow 2\frac{L^2}\sigma$, $z_0 \leftarrow y_0$\;
		\For{$k = 0, 1, \ldots, K-1$} {
			$\tau_k \leftarrow \frac2{(k+2)}$, $\ \eta_k \leftarrow \frac{(k+1)}{2\beta} $, $\ w_{k} \leftarrow (1-\tau_{k}) y_k + \tau_k z_k$\;
			Choose $x_{k+1}, y_{k+1}$ ensuring: 
			\begin{equation*}		
			g(x_{k+1}, y_{k+1}) = \min_x g(x, y_{k+1}), \ \ 
			y_{k+1} = \PY{w_k + \frac1\beta \nabla_y g(x_{k+1}, w_k)}
			\end{equation*}\\
			$z_{k+1} \leftarrow \PY{z_k + \eta_k \nabla_y g(x_{k+1}, w_k)}$, $\ \ \bar{x}_{k+1} \leftarrow \frac{2}{(k+1)(k+2)}\sum_{i=1}^{k+1} i\cdot x_i$
		}
		\KwRet{$\bar{x}_{K}, y_{K}$}\;
	}
	\caption{Conceptual Dual Implicit Accelerated Gradient (\cmpscc{}) for strongly-convex--concave programming}
	\label{algo:conceptual_minimax_stronglycvx_cve_agd}
\end{algorithm}
\textbf{Implementability}:
The first question is whether it is easy enough to implement such a step? It turns out that it is indeed possible to quickly find points $x_{k+1}$ and $y_{k+1}$ that approximately satisfy the above requirements. The reason is that:
\begin{itemize}
	\item Since $g(\cdot,y)$ is smooth and strongly convex for every $y \in \mathcal{Y}$, we can find $\epsilon$-approximate minimizer for a given $y$ in $\order{\log \frac{1}{\epsilon}}$ iterations.
	\item Let $x^*(y) \defeq \argmin_{x\in \mathcal{X}} g(x,y)$. The iteration $y^{i+1} = \PY{w_k + \frac{1}{\beta} \nabla_y g(x^*(y^i),w_k)}$ is a $1/{2}$-contraction with a unique fixed point satisfying the update step requirements (i.e., Step $4$ of Algorithm~\ref{algo:conceptual_minimax_stronglycvx_cve_agd}). See Lemma~\ref{lem:mp-step} in Appendix~\ref{sec:minimax_stronglycvx_cve_agd_proof} for a proof. This means that only $\order{\log \frac{1}{\epsilon}}$ iterations again suffice to find an update that approximately satisfies the requirements.
\end{itemize}

\noindent
\textbf{Convergence rate}:
Since $y_{k+1}$ and $z_{k+1}$ correspond to an AGD update for $g(x_{k+1},\cdot)$, we can use the potential function decrease argument for AGD (Lemma~\ref{lem:agd-pf} in Appendix~\ref{sec:nesterov}) to conclude that $\forall y \in \mathcal{Y}$,
\begin{align*}
	&(k+1)(k+2) \left(g(x_{k+1},y) - g(x_{k+1},y_{k+1})\right) + 2\beta \cdot \norm{y - z_{k+1}}^2 \\
	&\leq k(k+1) \left(g(x_{k+1},y) - g(x_{k+1},y_{k})\right) + 2\beta \cdot \norm{y - z_{k}}^2 \\
	&\leq k(k+1) \left(g(x_{k+1},y) - g(x_k,y)\right) + k(k+1) \left(g(x_{k},y) - g(x_{k},y_{k})\right) + 2\beta \cdot \norm{y - z_{k}}^2,
\end{align*}
where the last step follows from the fact that $x_{k} = \argmin_x g(x, y_{k})$ and so $g(x_{k},y_{k}) \leq g(x_{k+1},y_{k})$. Noting that we can further recursively bound $k(k+1) \left(g(x_{k},y) - g(x_{k},y_{k})\right) + 2\beta \cdot \norm{y - z_{k}}^2$ as above, we obtain
\begin{align*}
	&\quad (k+1)(k+2) \left(g(x_{k+1},y) - g(x_{k+1},y_{k+1})\right) + 2\beta \cdot \norm{y - z_{k+1}}^2 \\
	&\quad \leq k(k+1) g(x_{k+1},y) - \sum_{i=1}^k (2i)\cdot g(x_i,y) + 2\beta \cdot \norm{y - z_{0}}^2 \\
	&\Rightarrow \sum_{i=1}^{k+1} (2i)\cdot g(x_i,y) - (k+1)(k+2) g(x_{k+1},y_{k+1}) \leq 2\beta \cdot \norm{y - z_{0}}^2.
\end{align*}
Since $g(x_{k+1},y_{k+1}) \leq g(x,y_{k+1})$ for every $x\in\mathcal{X}$, we have
\begin{align*}
	&\sum_{i=1}^{k+1} (2i)\cdot g(x_i,y) - (k+1)(k+2) g(x,y_{k+1}) \leq 2\beta \cdot \norm{y - z_{0}}^2 \\
	&\Rightarrow g(\bar{x}_{k+1},y) - g(x,y_{k+1}) \leq \frac{2\beta \cdot \norm{y - z_{0}}^2}{(k+1)(k+2)},
\end{align*}
where $\bar{x}_{k+1} \defeq \frac{1}{(k+1)(k+2)}\sum_{i=1}^{k+1} (2i)\cdot x_i$.
Since $x$ and $y$ are arbitrary above, this gives a $\order{1/k^2}$ convergence rate for the primal dual gap.
\subsection{Error analysis}
The main issue with Algorithm~\ref{algo:conceptual_minimax_stronglycvx_cve_agd} is that the update step is not exactly implementable. However, as we noted in the previous section, we can quickly find updates that almost satisfy the requirements. Algorithm~\ref{algo:minimax_stronglycvx_cve_agd} presents this inexact version. The following theorem states our formal result and a detailed proof is provided in Appendix \ref{sec:minimax_stronglycvx_cve_agd_proof}. 
%A proof is provided in Appendix \ref{sec:outer-max-smooth_proof}.
%We introduce a mirror-prox algorithm for strongly-convex--concave minmax problem (\mpscc{}) which has a fast convergence rate of $\tilde{O}(\varepsilon^{-1/2})$. 
%dual problem, but with adopting potential-function based proof of \cite{bansal2017potential}
%A cursory glance of the \mpscc{} (Algorithm \ref{algo:minimax_stronglycvx_cve_agd}) reveals that it is a modified version of accelerated gradient descent \cite[Eqns. 5.47-5.49]{bansal2017potential} on some function of $y$ with a modified step given by \mpstep{}, which is inspired from the conceptual prox-method of \cite{nemirovski2004prox}. In the following lemma we analyze the \mpstep{} sub-routine, which is the most non-trivial step of the algorithm. 

\begin{algorithm}[t]
	\DontPrintSemicolon % Some LaTeX compilers require you to use \dontprintsemicolon instead
	\SetKwFunction{FMPSCC}{MP-SCC}
	\SetKwFunction{FMPSTEP}{Imp-STEP}
	\KwIn{$g$, $L$, $\sigma$, $x_0$, $y_0$, $K$, $\{\varepsilon_{\rm step}^{(k)}\}_{k=1}^K$}
	\KwOut{$\bar{x}_{K}, y_{K}$}
	\SetKwProg{Fn}{}{:}{}
%	\Fn{\FMPSCC{$g$, $L$, $\sigma$, $x_0$, $y_0$, $K$, $\{\varepsilon_{\rm step}^{(k)}\}_{k=1}^K$}}
	{
	Set $\beta \leftarrow 2\frac{L^2}\sigma$, $z_0 \leftarrow y_0$\;
	\For{$k = 0, 1, \ldots, K-1$} {
		$\tau_k \leftarrow \frac2{(k+2)}$, $\ \eta_k \leftarrow \frac{(k+1)}{2\beta} $, $\ w_{k} \leftarrow (1-\tau_{k}) y_k + \tau_k z_k$\;
		$x_{k+1}, y_{k+1} \leftarrow $ \FMPSTEP{$g$, $L$, $\sigma$, $x_0$, $w_k$, $\beta$, $\varepsilon^{(k+1)}_{\rm step}$}, ensuring: 
		\begin{equation*}		
		g(x_{k+1}, y_{k+1}) \leq \min_x g(x, y_{k+1}) + \varepsilon^{(k+1)}_{\rm step}, \ \ 
		y_{k+1} = \PY{w_k + \frac1\beta \nabla_y g(x_{k+1}, w_k)}
		\end{equation*}\\
		$z_{k+1} \leftarrow \PY{z_k + \eta_k \nabla_y g(x_{k+1}, w_k)}$, $\ \ \bar{x}_{k+1} \leftarrow \frac{2}{(k+1)(k+2)}\sum_{i=1}^{k+1} i\cdot x_i$
	}
	\KwRet{$\bar{x}_{K}, y_{K}$}\;
	}
	%	\caption{Mirror-prox for strongly-convex--concave programming}
	%	\label{algo:minimax_stronglycvx_cve_agd}
	%\end{algorithm}
	\vspace{0.25cm}
	%\begin{algorithm}[]
	%	\DontPrintSemicolon % Some LaTeX compilers require you to use \dontprintsemicolon instead
	%	%\KwIn{$g$, $G$, $\sigma$, $w$, $\varepsilon$}
	%	\SetKwFunction{FMPSTEP}{MP-STEP}
	%	\SetKwProg{Fn}{}{:}{}
	\Fn{\FMPSTEP{$g$, $L$, $\sigma$, $x_0$, $w$, $\beta$, $\varepsilon_{\rm step}$}}{
		Set $\varepsilon_{\rm mp} \leftarrow \frac{2 \sigma}{5 L}\sqrt{\frac{2 \varepsilon_{\rm step}}{L}}$, $R \leftarrow \ceil{\log_2 \frac{2D_\cY}{\varepsilon_{\rm mp}}}$, $\varepsilon_{\rm agd} \leftarrow \frac{\sigma \beta^2 \varepsilon_{\rm mp}^2}{32 L^2}$, 
		$y_0 \leftarrow w$\;
		\For{$r = 0, 1, \ldots, {R}$} {
			Starting at $x_0$ use  AGD (Algorithm~\ref{algo:AGD} with $-g(\cdot, y_r)$)  to compute $x_r$ such that: % $\sigma$-strongly convex $L$-smooth function \cite{nesterov1983method} find $\hat{x}_r$ such that,
			\begin{equation}
			g(\hat{x}_r, y_r) \leq \min_x g(x, y_r) + \varepsilon_{\rm agd},\ \
			%y_{r+1} \leftarrow \PY{w + \frac1\beta \nabla_y g(\hat{x}_{r}, w)}  
			\label{eq:mp-step-agd-error} %label{eq:mp-step-y-step}
			\end{equation}\\
%			\vspace{-2.em}
%			\begin{equation}
			$y_{r+1} \leftarrow \mathcal{P}_{\cY} \big(w + \frac1\beta \nabla_y g(\hat{x}_{r}, w)\big)$ %
%			\label{eq:mp-step-y-step}
%			\end{equation}
		}
		\KwRet{$\hat{x}_{R}$, $y_{R+1}$}
	}
	%\caption{Mirror-prox step}
	%\label{algo:mirror_prox_step}

	\caption{Dual Implicit Accelerated Gradient (\mpscc{}) for strongly-convex--concave programming}
	\label{algo:minimax_stronglycvx_cve_agd}
\end{algorithm}

\begin{theorem}[Convergence rate of  \mpscc]
	\label{thm:minimax_stronglycvx_cve_agd}
	Let $g: \cX \times \cY \to \reals$  be a $L$-smooth, $\sigma$-strong-convex--concave function on  $\cX = \reals^p$ and a convex compact sub-set $\cY \subset \reals^q$. Then, after $K$ iterations, \mpscc (Algorithm \ref{algo:minimax_stronglycvx_cve_agd}) finds $(\bar{x}_{K}, y_K)$ s.t.: 
	\begin{align}
	\max_{\ty \in \cY} g(\bar{x}_{K}, \ty) -  \min_{\tx \in \cX} g(\tx, y_K) 
	&\leq \frac{4\frac{L^2}{\sigma} D_{\cY}^2 + \sum_{k=1}^K k(k+1)\,\varepsilon_{\rm step}^{(k)}}{K(K+1)}\,.
	\end{align}
	In particular, 
%	when $\varepsilon^{(k)}_{\rm step} = \frac1{k^2 (k+1)}$ we have,
%	\begin{align}
%	\max_{\ty \in \cY} g(\bar{x}_{K}, \ty) -  \min_{\tx \in \cX} g(\tx, y_K) 
%	&\leq \frac{4\frac{L^2}{\sigma} D_{\cY}^2 + \ln K + 1}{K(K+1)}\, \text{, and, }
%	\end{align}
%	and 
	setting $\varepsilon^{(k)}_{\rm step} = \frac{{L^2}D_{\cY}^2}{{\sigma} k^3 (k+1)}$ we have: $\max_{\ty \in \cY} g(\bar{x}_{K}, \ty) -  \min_{\tx \in \cX} g(\tx, y_K) \leq \frac{6\frac{L^2}{\sigma} D_{\cY}^2}{K(K+1)}$. 
%	\begin{align}
%	\max_{\ty \in \cY} g(\bar{x}_{K}, \ty) -  \min_{\tx \in \cX} g(\tx, y_K) 
%	&\leq \frac{6\frac{L^2}{\sigma} D_{\cY}^2}{K(K+1)}\,,
%	\end{align}
Furthermore, for this setting the total first order oracle complexity is given by: ${O}(\sqrt{\frac{L}\sigma} K \log^2 (K))$.%	with total number of gradient computations per iteration being at most be $K \ceil[\bigg]{\log_2 5 K^2 \sqrt{\frac{L}\sigma}} \, 	O\Big(\sqrt{\frac{L}\sigma} \log \big(K^4 \big)\Big) = O(\sqrt{\frac{L}\sigma} K \log^2 (K))$.	
\end{theorem}

\bigskip
\noindent
{\bf Remark 1}: Theorem~\ref{thm:minimax_stronglycvx_cve_agd} shows that~\mpscc needs $\tilde{O}(({L}/{\sigma})\cdot ({\sqrt{L}D_{\cY}}/{\sqrt{\varepsilon}}))$ gradient queries for finding a $\varepsilon$-primal-dual-pair, while current best-known rate is $O({1}/{\varepsilon})$ achieved by Mirror-Prox. 
%\textcolor{blue}{
This dependence in $\varepsilon$ and $D_{\cY}$ is optimal, as 
it is shown in \cite[Theorem 10]{ouyang2018lower} that $\Omega(D_{\cY} (L-\sigma) / \sqrt{\sigma \varepsilon})$ gradient queries are  necessary to achieve $\varepsilon$ error in the primal-dual gap.
%}

\bigskip
\noindent
{\bf Remark 2}: Unlike standard AGD for $h(y)$, which only updates $y_k$ in the outer-loop, \mpscc's outer-step updates both $x_k$ and $y_k$ thus allowing us to better track the primal-dual gap. However, \mpscc's dependence on the condition number ${L}/{\sigma}$ seems sub-optimal and can perhaps be improved if we do not compute \mpstep nearly optimally allowing for inexact updates; we leave further investigation into improved dependence on the condition number for future work. 

%
%\begin{definition}\label{def:smooth2}
%	A function $g(x,y)$ is said to be $(L_{xx},L_{xy},L_{yx},L_{yy})$-smooth if: 
%	\begin{eqnarray*}
%	\norm{\nabla_x g(x,y) - \nabla_x g(x',y')}  
%	&\leq& L_{xx} \norm{x - x'}  + L_{xy} \norm{y - y'}  \;, \text{ and }\\ 
%	\norm{\nabla_y g(x,y) - \nabla_y g(x',y')} &\leq& 
%	L_{yx} \norm{x - x'} + L_{yy} \norm{y - y'}\;, 
%	\end{eqnarray*}
%	for all $x,x',y$, and $y'$.
%\end{definition}
%
%
%\bigskip
%\noindent
%{\bf Remark 3 (Comparisons to algorithms requiring exact optimization of sub-problems)}: 
%\cite{hamedani2018primal} studies a minimax optimization with a special structure of the form:  
%\begin{eqnarray} 
%	\min_{x\in\cX} \; \max_{y\in\cY} \;  g(x,y) \;=\; \phi(x) + \langle \psi(x) , y \rangle - \xi(y)  \;.
%\end{eqnarray} 
%Is it assumed  that $\phi(x)$ is $\sigma$-strongly convex, 
%$\psi(x)$ is convex, differentiable, $L_{yx}$-Lipschitz, and $L_{xx}$-smooth, 
%and an oracle access to the exact solution to the following sub-routines:  
%$\arg\min_{x\in\cX} \{t \phi(x) + \langle s ,x \rangle + \|x-x'\|^2 \} $ for all $t,s,x'$, and  
%$\arg\min_{y\in\cY} \{t\xi(x) + \langle s,y\rangle + \|y-y'\|^2\}$ for all $t,s,y'$. 
%It is shown that after $K$ accesses to such an oracle, the algorithm achieves 
%a primal-dual gap $\max_{\ty \in \cY} g(\bar{x}_{K}, \ty) -  \min_{\tx \in \cX} g(\tx, y_K) $ 
%upper bounded by $(L_{xx}+L_{yx}^2)D_{\cX}^2/K^2$. 

%We now study the non-convex-concave minimax problem and show how the result of this section can help improve the result for the challenging non-convex setting as well. 

% !TEX root = ./smooth_minimax.tex

\section{Nonconvex concave saddle point problem}\label{sec:noncvx}

%\begin{align}
%\min_{x \in \cX} \; [\; f(x) =  \max_{y \in \cY} g(x, y) \;]\;, 
%\tag{P2}\label{prob:smooth-minimax}
%\end{align}
%where
We study the nonconvex concave minimax problem~\eqref{eqn:minimax}
where $g(x, \cdot)$ is concave, $g(\cdot, y)$ is nonconvex, and $g(\cdot, \cdot)$ is $L$-smooth, $\cX = \reals^p$ (such that ${\rm Proj}_{\cX}(x) = x$) and $\cY$ is a convex compact sub-set of $\reals^q$. 
%As the problem is nonconvex, the minimum primal-dual gap is unknown, 
%making it challenging to track its convergence.  
As mentioned in Section~\ref{sec:background},
we measure the convergence to an approximate FOSP of this problem (see Definition~\ref{def:eps_fosp}) but it requires weak-convexity of $f(x) \defeq \max_{y \in \cY} g(x,y)$. The following lemma guarantees weak convexity of $f$ given smoothness of $g$.% to 
%$L$-weak convexity of $f$. %, allowing us to track convergence of the gradient of its Moreau envelope as discussed in Section \ref{sec:nonconvex-concave-minmax}. 
\begin{lemma}\label{lem:weakly-cvx-smooth}
Let $g(\cdot, y)$ be continuous and $\cY$ be compact. Then $f(x) = \max_{y \in \cY} g(x, y)$ is $L$-weakly convex, if $g$ is $L$-weakly convex in $x$ (Definition \ref{def:smooth}), or if $g$ is $L$-smooth in $x$ .
\end{lemma}
See Appendix \ref{sec:proof_weakly-cvx-smooth} for the proof. % of the lemma.
The arguments of~\cite{jin2019minmax} easily extend to show that applying subgradient method on $f(x)$,~\cite{davis2018stochastic} gives a convergence rate of $\order{1/k^{1/5}}$.
%it would be difficult to exploit the smoothness structure of $g$ and hence suffer from slow rates.
Instead, we exploit the smooth minimax form of $f(\cdot)$ to design a faster converging scheme.
The main intuition comes from the proximal viewpoint that gradient descent can be viewed as iteratively forming and optimizing local quadratic upper bounds.
%, we can construct a sequence of explicit quadratic approximations of $f$ and optimize them to converge to a desirable solution.
As $f$ is weakly convex, adding enough quadratic regularization should ensure that the resulting sequence of problems are all strongly-convex--concave. We then exploit \mpscc to efficiently solve such local quadratic problems to obtain improved convergence rates.
% for convergence to $\varepsilon$-FOSP. 
%\if0
%Our strategy is to design local updates involving 
%{\em strongly-convex concave} problems, 
%such that we can apply \mpscc. 
%The fast convergence rate of \mpscc in Theorem~\ref{thm:minimax_stronglycvx_cve_agd} 
%will translate into faster convergence rate to the $\varepsilon$-FOSP (Definition \ref{def:eps_fosp}). 
%To this end, at each step of algorithm, we search for the minimizer $\hat{x}_\lambda(x_k)$ of the Moreau envelop %(Lemma~\ref{lem:moreau-properties}) 
%with parameter $\lambda=1/{(2L)}$, at $x_k$. 
%Together with $L$-weak convexity of $g(\cdot ,\cdot)$ of Lemma \ref{lem:weakly-cvx-smooth}, 
%this ensures the desired strong convexity (Lemma~\ref{lem:moreau-properties}). 
%\fi
Concretely, let
\begin{eqnarray}
\wf(x; x_k)  \;\; = \;\; \max_y g(x, y)+ L \| x - x_k \|^2 \;. \label{eq:gradmap_general}
\end{eqnarray}
By $L$-weak-convexity of $f$, $\wf(x;x_k)$ is {\em strongly}-convex--concave (Lemma \ref{lem:weak-plus-strong-cvx}) that can be solved using \mpscc up to {\em certain accuracy} to obtain $x_{k+1}$. We refer to this algorithm as  \lqancc and provide a pseudo-code for the same in Algorithm~\ref{algo:minimax_noncvx_cve_gd}.
%The inner loop of \lqancc forms the quadratic approximation problem and solves it using \mpscc up to a fixed accuracy, to ensure overall fast rate of convergence. 
%Then minimizer of the above function would be the minimizer for the inner minimization of Moreau envelope function $\hat{x}_{\frac1{2L}}(x_k) = \arg\min_{x \in \cX} f(x) + L \|x_k-x\|^2$. 
%to find an $\varepsilon$-FOSP  of the  non-smooth minimization problem $\min_{x \in \cX} f(x)$. 
%which is characterized by the gradient of its Moreau envelope.  
%By Lemma \ref{lem:weakly-cvx-smooth}, for any fixed $y$ ($\cY= \{y\}$), $g(\cdot, y)$ is $L$-weakly convex. Also note that the quadratic term in $g(x, y) + 
%We introduce a new approach to solve the 
%non-convex concave minimax problem of 
%\eqref{prob:smooth-minimax} that we refer to as \lqancc. 
%\lqancc is designed to iteratively solve the above local strong convex approximations. 
%We propose using  \mpscc{} to find an approximate solution to this inner problem.  
% (Algorithm \ref{algo:minimax_stronglycvx_cve_agd}) algorithm for strongly-convex--concave min-max problem to set $x_{k+1}$ as an approximate minimizer of the function $f(x; x_k)$. In the following theorem, we prove that iteratively solving $f(x; x_k)$ converges to an $\varepsilon$-FOSP of \eqref{prob:smooth-minimax}. 
\begin{algorithm}[]
	\DontPrintSemicolon % Some LaTeX compilers require you to use \dontprintsemicolon instead
	\KwIn{$g$, $L$, $\varepsilon$, $x_0$, $y_0$}
	\KwOut{$x_k$}
	Set $\tepsilon \leftarrow \frac{\varepsilon^2}{64\,L}  $\;
	\For{$k = 0, 1, \ldots, K$} {
		Using \mpscc for strongly convex concave minimax problem, find $x_{k+1}$ such that,
		\begin{align}
		\max_{y \in \cY} g(x_{k+1}, y)+ L \| x_{k+1} - x_k \|^2 \; \leq \; \min_x \max_{y \in \cY} g(x, y)+ L \| x - x_k \|^2 + \frac\tepsilon4
		\end{align}
		\If{$\max_{y\in\cY} g(x_k, y) - \frac{3\tepsilon}4 \leq \max_{y \in \cY} g(x_{k+1}, y)+ L \| x - x_k \|^2$}{
			\KwRet{$x_{k}$}
		} 
	}
	\caption{Proximal Dual Implicit Accelerated Gradient (\lqancc) for nonconvex concave programming}
	\label{algo:minimax_noncvx_cve_gd}
\end{algorithm}
The following theorem gives convergence guarantees for \lqancc.
%states our formal result for the non-convex-concave minimax problem. %and a formal proof is provided in Appendix \ref{sec:minimax_noncvx_cve_gd_proof}. 
\begin{theorem}[Convergence rate of \lqancc]
Let $g(x,y)$ be $L$-smooth, $g(x, \cdot)$ be concave,
%, $g(\cdot, u)$ be nonconvex but $L$-smooth, 
$\cX$ be $\reals^p$, $\cY$ be a convex compact subset of $\reals^q$, and the minimum value of function $f(x) = \max_{y \in \cY} g(x, y)$ be  bounded below, i.e.~$f(x)\geq f^*>-\infty$. Then \lqancc (Algorithm \ref{algo:minimax_noncvx_cve_gd}) after, $$K = \ceil[\bigg]{\frac{4^4 L (f(x_0) -f^*)}{3\varepsilon^2}}$$ steps outputs an $\varepsilon$-FOSP. The total first-order oracle complexity to output $\varepsilon$-FOSP is:\\ $O\big({\frac{L^2 D_{\cY}(f(x_0) -f^*)}{\varepsilon^3}} \log^2 \big(1/{\varepsilon}\big)\big)\,.$%number of gradient computations done per iteration by the inner strongly-convex--concave min-max step is at most $O\big(({L D_{\cY}}/\varepsilon) \log^2 \big(1/{\varepsilon}\big)\cdot \big)\,.$
\label{thm:minimax_noncvx_cve_gd}
\end{theorem}
Note that \lqancc solves the quadratic approximation problem to higher accuracy of $O(\epsilon^2)$ which then helps bounding the gradient of the Moreau envelope. Also due to the modular structure of the argument, a faster inner loop for special settings, e.g., when $g(x,y)$ is a finite-sum, can ensure more efficient algorithm. While our algorithm is able to significantly improve upon existing state-of-the-art rate of $O(1/\varepsilon^5)$ in general nonconvex-concave setting \cite{jin2019minmax}, it is unclear if the rate can be further improved. In fact, precise lower-bounds for this setting are mostly unexplored and we leave further investigation into lower-bounds as a topic of future research. 

\begin{proof}
	We first note that by Lemma \ref{lem:weak-plus-strong-cvx} and $L$-weak convexity of $g(\cdot, y)$ and $2L$-strong convexity of $L\|x-x_k\|^2$, $\wg(x,y; x_k) := g(x, y) + L \| x - x_k \|^2$ is $L$-strongly-convex. Similarly, $\wf(\cdot; x_k):=\max_{y \in \cY} [\wg(x,y; x_k) = g(x, y) + L \| x - x_k \|^2]$ is also $L$-strongly-convex.

\if0
	Similarly, 
Let,
\begin{align}\label{eq:quad-approx-smooth}
f(x; x_k) = \max_{y \in \cY} [g(x,y; x_k) = g(x, y) + L \| x - x_k \|^2 ]\,.
\end{align}
Then by Lemma \ref{lem:weak-plus-strong-cvx} and $L$-weak convexity of $g(\cdot, y)$ and $2L$-strong convexity of $L\|x-x_k\|^2$, we get  that $g(\cdot, y; x_k)$ is $L$-strongly-convex.
Similarly, $f(\cdot; x_k)$ can be shown to be $L$-strongly convex too, since $f(x, x_k) = [\max_y g(x,y)] + L\|x- x_k \|^2$ and $\max_u g(x, u)$ is $L$-weakly convex by Lemma \ref{lem:weakly-cvx-smooth}.
\fi

We now divide the analysis of each iteration of our algorithm into two cases:  

\noindent
{\bf Case 1: 
	%{\bf If $\mathbf{f(x_{k+1}; x_k) {\boldsymbol\leq}  f(x_k) - 3\boldsymbol{\tepsilon}/4}$},  
	${\wf(x_{k+1}; x_k) {\leq}  f(x_k) - 3{\tepsilon}/4}$.}  As every instance of Case 1 ensures $f(x_{k+1})\leq \wf(x_{k+1}; x_k) \leq f(x_k) - 3{\tepsilon}/4$,  we can have only 
$\ceil[\Big]{\frac{4(f(x_0) - f^*)}{3\tepsilon}}$ Case 1 steps before termination. This claim requires monotonic decrease in $f(x_k)$ which holds until $f(x_{k+1})\geq f(x_k)$, after which $\wf(x_{k+1}; x_{k})\geq f(x_k)$, which in-turn imply that \lqancc terminates (see termination condition of \lqancc).

\noindent 
{\bf Case 2: 
	%{\bf If $\mathbf{\wf(x_{k+1}; x_k) \boldsymbol>  f(x_k) - 3\boldsymbol{\tepsilon}/4}$}
	${\wf(x_{k+1}; x_k) {>}  f(x_k) - 3{\tepsilon}/4}$:} In this case, we show that $x_k$ is already an $\varepsilon$-FOSP and  the algorithm returns $x_k$. 
\begin{align}
f(x_k) - \frac{3\tepsilon}4 < \wf(x_{k+1}; x_k) \leq \min_x \wf(x; x_k)  + \frac{\tepsilon}4  
\;\;\;\; \implies   \;\;\;\; f(x_k)  < \min_x \wf(x; x_k)  + {\tepsilon}
\label{eq:less-decrease-smooth}
\end{align}
Define $x_k^*$ as the point satisfying  $x^*_k = \arg\min_x \wf(x; x_k)$. 
By $L$-strong convexity of $\wf(\cdot; x_k)$ \eqref{eq:gradmap_general}, we prove that $x_k$ is close to $x_k^*$:{\small 
\begin{align}
\wf(x^*_k;x_k) + \frac{L}2 \|x_k-x^*_k\|^2 \; \leq \; \wf(x_{k};x_k) \; =\;  f(x_k) \; \overset{(a)}<  \; \wf(x^*_{k};x_k) + \tepsilon  
\implies\|x_k-x^*_k\| < \sqrt{\frac{2\tepsilon}{L}}  \label{eq:moreau-ball-ub-smooth}
\end{align}}
where $(a)$ uses \eqref{eq:less-decrease-smooth}. Now consider any $\tx \in \cX$, such that $4\sqrt{\tepsilon/L} \leq \| \tx - x_k\|$. Then,
\begin{align}
f(\tx) + L \|\tx - x_k\|^2 &= \max_{y \in \cY} g(\tx, y) + L \|\tx - x_k\|^2 = \wf(\tx; x_k) \overset{(a)}= \wf(x_k^*; x_k) + \frac{L}2 \|\tx - x^*_k\|^2 \nonumber \\
&\overset{(b)}\geq f(x_k) - \tepsilon + \frac{L}2 (\|\tx - x_k\| - \|x_k - x^*_k\|)^2 \overset{(c)}\geq  f(x_k) + \tepsilon, \label{eq:moreau-ball-lb-smooth}
\end{align}
where $(a)$ uses uses $L$-strong convexity of $\wf(\cdot; x_k)$ at its minimizer $x^*_k$, $(b)$ uses \eqref{eq:less-decrease-smooth}, and $(b)$ and $(c)$ use triangle inequality, \eqref{eq:moreau-ball-ub-smooth} and $4\sqrt{\tepsilon/L} \leq \| \tx - x_k\|$.

Now consider the Moreau envelope, $f_{\frac1{2L}}(x) = \min_{x' \in X} \phi_{\frac1{2L}, x}(x')$ where $\phi_{\lambda,x}(x') = f(x') + L \|x - x'\|^2$. Then, we can see that $\phi_{\frac1{2L}, x_k}(x')$ achieves its minimum in the ball $\{x' \in \cX \,|\, \|x' - x_k\| \leq 4\sqrt{\tepsilon/L} \}$ by \eqref{eq:moreau-ball-lb-smooth} and Lemma \ref{lem:moreau-properties}(a). Then, with Lemma \ref{lem:moreau-properties}(b,c) and $\tepsilon = \frac{\varepsilon^2}{64\,L}$, we get that,
\begin{align}
\| \nabla f_{\frac1{2L}}(x_k) \| \leq (2L) \| x_k - \hat{x}_{\frac1{2L}}(x_k) \| = 8 \sqrt{L \tepsilon} = \varepsilon,
\end{align}
i.e., $x_k$ is an $\varepsilon$-FOSP. 

By combining the above two cases, we establish that $O\big(\ceil[\big]{\frac{4(f(x_0) - f^*)}{3\tepsilon}}\big)$ ``outer" iterations ensure convergence to a $\varepsilon$-FOSP. We now compute the first-order complexity of each of these ``outer" iterations. Recall that we use use the \mpscc{} (Algorithm \ref{algo:minimax_stronglycvx_cve_agd}) algorithm for $L$-strongly-convex concave $2L$-smooth minimax problem to solve the inner optimization problem. So, if for each iteration of inner problem, \mpscc{} algorithm takes $K$ steps then, by $\tepsilon=\frac{\varepsilon^2}{64\,L}$ and Theorem \ref{thm:minimax_stronglycvx_cve_agd},
\begin{align}
\frac{6(2L)^2 D^2_{\cY}}{L K^2} \leq \frac{\tepsilon}{4} = \frac{\varepsilon^2}{2^8 L} \;\implies\; O\bigg(\frac{L D_{\cY}}{\varepsilon}\bigg) \leq K
\end{align}
Therefore the number of gradient computations required for each iteration of inner problem is $ O\Big(\frac{L D_{\cY}}{\epsilon} \log^2 \Big(\frac{1}{\varepsilon} \Big) \Big)$ (Theorem \ref{thm:minimax_stronglycvx_cve_agd}), which along with the bound on the number of outer iterations establishes the Theorem's upper bound on the number of first-order oracle calls.% now follows.% we obtain the upper bound on the number of first-order oracle complexity. 
%see that the total number of gradient computation to reach $\varepsilon$-FOSP is: 
%\begin{align}
%\ceil[\Bigg]{\frac{4^4 L (f(x_0) - f^*)}{3\varepsilon^2}} O\bigg(\frac{L D_{\cY}}{\epsilon} \log^2 \bigg(\frac{1}{\varepsilon} \bigg) \bigg)\,.
%\end{align}
\end{proof}

%!TEX root = ./smooth_minimax.tex
\subsection{Minimizing finite max-type function with smooth components}
As a special case of nonconvex--concave minimax problem, consider minimizing a weakly convex $f(x)$, with a special structure of {\em finite max-type function}: 
\begin{align} %\label{eq:finite-minimax}
	\min_x \; \Big[f\, (x) =  \max_{1 \leq i \leq m} f_i(x) \,\Big]  
	\tag{P3}\label{prob:finite-minimax}\;,
\end{align}
where $x \in \reals^p$, the functional components $f_i(x)$'s could be {\em nonconvex} but are $L$-smooth and $G$-Lipschitz. Suppose $f$ itself takes a minimum value $f^* > -\infty$.
%Recall that using Lemma \ref{lem:weakly-cvx-smooth} and smoothness of $f_i(x)$'s, $f(x)$ is weakly-convex, and the goal is to find an approximate FOSP (Definition~\ref{def:eps_fosp}). %Now using appropriate mirror-map in Algorithm~\ref{algo:minimax_noncvx_cve_gd}, we can get result similar to the one in Corollary
%This is a special case of the the non-convex min-max problem \ref{prob:smooth-minimax}, as we can re-write $f$ as $f(x) = \max_{0\preceq y\preceq 1, \sum_{i=1}^m y_i=1} \sum_i y_i \cdot f_i(x)$.
%max_{y \in \simplex_m} g(x, y)$ where $g(x, y) = \sum_{i \in [m]} y_i f_i(x)$, where $\simplex_m  = \{y \in [0, 1]^m | \sum_{i \in [m]} y_i = 1 \}$.
%While smoothness is not preserved when we take a finite max of smooth functions, the weak convexity (Definition~\ref{def:weak-convex}) is indeed preserved  (Lemma \ref{lem:weakly-cvx-smooth}). That is, the smoothness of $f_i(x)$'s implies weakly convexity of $f_i(x)$'s, which implies weak convexity of $f(x)$. 
%\begin{lemma}\label{lem:weakly-cvx}
%$f(x) = \max_{i \in [m]} f_i(x)$ is $L$-weakly convex, if all $f_i(x)$'s are $L$-smooth.
%\end{lemma}
%So, for this weakly convex problem, the goal is to find an 
%approximate FOSP (Definition~\ref{def:eps_fosp}). 
%Ignoring the max-type structure, 
%one could apply 
%proximal stochastic subgradient method of \cite{davis2018stochastic} 
%designed for any non-smooth weakly-convex minimization. 
%This ensures that $\varepsilon$-FOSP can be achieved in $O(1/\varepsilon^4)$ operations. 
For this problem, we propose and study a proximal (\lqafncc) algorithm (Algorithm~\ref{algo:minimax_gd} presented in Appendix~\ref{sec:proof_finite_minimax_gd}) that is inspired by Algorithm~\ref{algo:minimax_noncvx_cve_gd} with the inner problem-solver replaced by Nesterov's finite convex minimax scheme \cite[Section 2.3.1]{nesterov1998introductory} instead of Algorithm~\ref{algo:minimax_stronglycvx_cve_agd}. Using same proof technique as Theorem~\ref{thm:minimax_noncvx_cve_gd}, we get: 
\begin{corollary}[Convergence rate of \lqafncc]
	If the functional components $f_i(x)$'s are $G$-Lipschitz and $L$-smooth,  and the optimal solution is 
	bounded below, i.e.~$f(x)\geq f^* > -\infty$, then after: $K=\ceil[\bigg]{\frac{4^4 L (f(x_0) - f^*)}{3 \varepsilon^2}}$ 	{\em outer} steps,  \lqafncc outputs an $\varepsilon$-FOSP. The total first-order oracle complexity to find $\varepsilon$-FOSP is:  $\ceil[\bigg]{\frac{4^4 L (f(x_0) - f^*)}{3 \varepsilon^2}} \cdot \ceil[\bigg]{\frac{2^4 G}{\varepsilon} (m \log^{3/2} m) {{}{}}}$. %	total iterations by the inner excessive gap technique step.
	\label{thm:finite_minimax_gd}
\end{corollary}
See Appendix \ref{sec:proof_finite_minimax_gd} for a proof.
%Current best rate is $O(m/\varepsilon^5)$ oracle calls~\cite{jin2019minmax}.
%{\color{blue}
Current best rate for this problem is achieved by subgradient methods.  
As the subgradient of a finite minimax function $\nabla_{i^*} f(x)$ is easy to evaluate, where $i^* \in \arg\max_{i} f_i(x)$, 
a rate of $O(m/\varepsilon^4)$ first-order oracle and function calls  is achieved by the state-of-the-art subgradient method in \cite{davis2018stochastic}. 
% Notice that, due to the finite minimax structure, a sub-gradient could be computed easily as $\nabla_{i^*} f(x)$, where $i^* \in \arg\max_{i} f_i(x)$.
%}
We can obtain a similar result using Algorithm~\ref{algo:minimax_stronglycvx_cve_agd} but it requires extension to non-Euclidean settings with the framework of Bregman divergences. This is fairly standard and will be updated in the next version of the paper.
\section{Experiments}\label{sec:experiments}
We empirically verify the performance of \lqafncc (Algorithm \ref{algo:minimax_gd}) on a synthetic finite max-type nonconvex minimization problem (\ref{prob:finite-minimax}). 
% and compare it against the sub-gradient descent \cite{davis2018stochastic}, which achieves a rate of $\tilde{O}(m/\varepsilon^3)$. 
 We consider the following problem.
\begin{align}
	\min_{x \in \reals^2}  \big[ f(x) = \max_{1 \leq i \leq m=9} f_i(x) \big]
\end{align}
where $f_i(x) = q_{(-1,\, (X^{(1)}_i, X^{(2)}_i),\, C_i)}(x)$  for all $1 \leq i \leq 8$, where $q_{(a, b, c)}(x) = a\|x - b\|_2^2 + c$, 
$X^{(1)}_i$ and $X^{(2)}_i$ are generated from the interval $[-3.0, 3.0]$ uniformly at random,
 and $C_i$ is generated from the interval $[1.0, 5.0]$ uniformly at random.  
 We fix the last component $f_9(x) = q_{(0.5,\, (0, 0),\,0)}(x)$.  
 Each $f_i$ is smooth with parameter $L=1$, which implies that $f$ is $L$-weakly convex.

We implement three algorithms: 
\lqafncc (Algorithm \ref{algo:minimax_gd}), \alqafncc (Algorithm \ref{algo:minimax_gd_adapt}), and subgradient method \cite{davis2018stochastic}. In \lqafncc, we use excessive gap technique \cite[Problem (7.11)]{nesterov2005excessive} (a primal-dual algorithm) to solve the inner sub-problem. 
As the stopping criteria $\wf(x_{k+1}; x_k) \leq \min_x \wf(x; x_k) + \tepsilon/4$ cannot be directly checked, 
we instead check a sufficient condition; we stop the excessive gap technique 
when the primal-dual gap is less than $\tepsilon/4$, which can be checked efficiently.  
\alqafncc is a variant of \lqafncc, where we adaptively and successively decrease the tolerance parameter $\varepsilon'$ starting from a large tolerance $\varepsilon_0$. It has the same first-order oracle complexity guarantee as \lqafncc (up to an $O(\log(1/\varepsilon))$ factor). 
% Theorem \ref{thm:adap_finite_minimax_gd}. 
However, in Figure~\ref{fig:finite_minimax_moreau}, we observe that \alqafncc can converge faster in practice. 
We set the initial tolerance $\varepsilon_0$ as $10.0$. 
For a description of the algorithm we refer to Appendix \ref{sec:adaptive_finite_minimax_gd}.

\begin{figure}[h]
	\centering
	\includegraphics[width=0.6\linewidth]{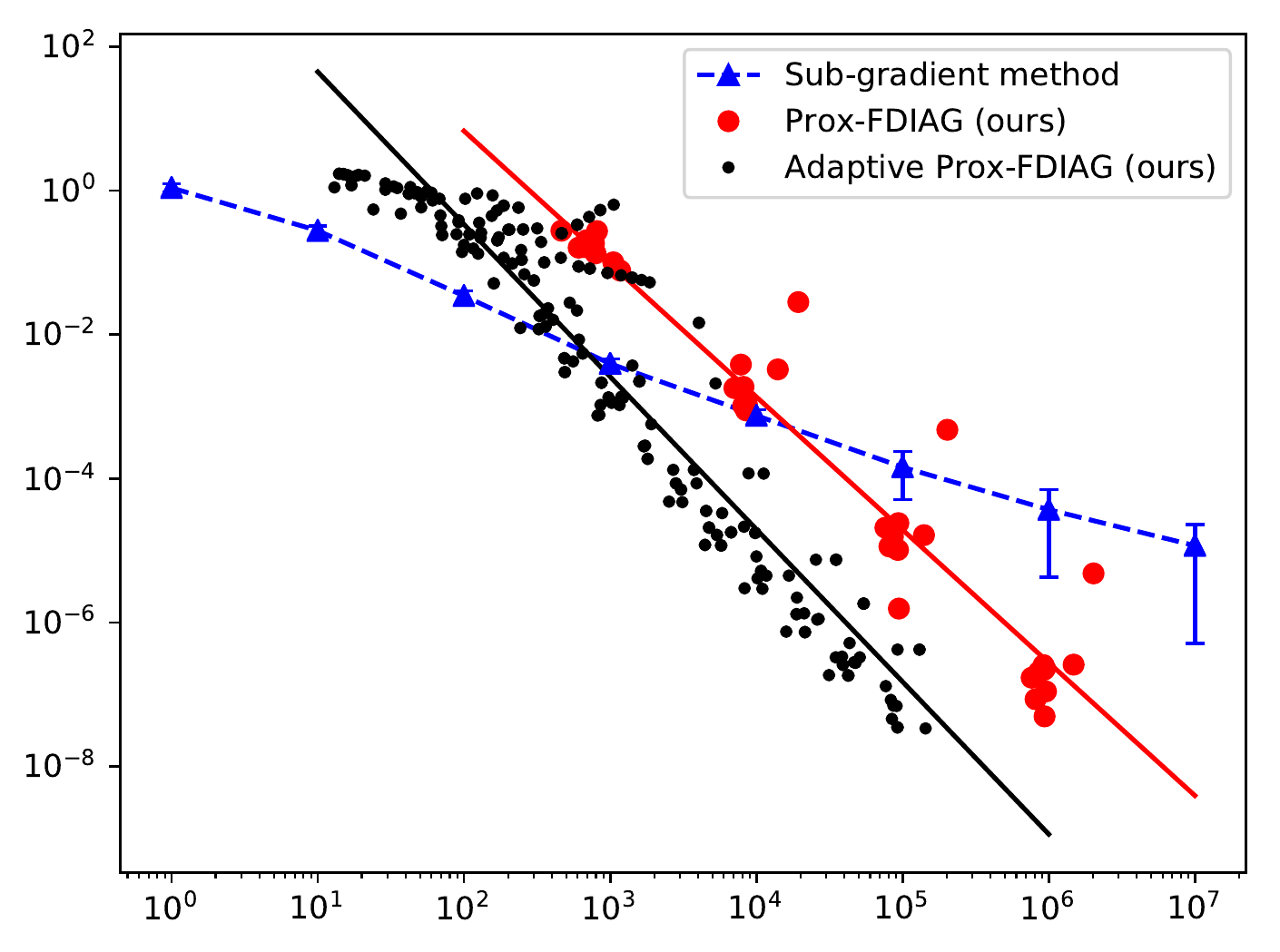}
	\put(-345, 100){{Norm of the gradient}}
	\put(-345, 85){{of Moreau envelope}}
	\put(-335, 69){{$\| \nabla f_{\frac1{2L}}(x_k)\|_2$}}
%	\put(-200, -10){{Number of inner most iterations ($k$)}}
	\put(-200, -10){{number of gradient oracle accesses $k$ }}
	\caption{%norm of gradient of Moreau envelope ($\| \nabla f_{\frac1{2L}}(x_k)\|_2$) against the number of iterations $k$  
	%for \lqafncc (ours), \alqafncc (ours), and the sub-gradient descent methods. 
	For small target accuracy $\varepsilon$ regime, \alqafncc (ours) has the fastest convergence rate followed by \lqafncc (ours) and subgradient method.}
	\label{fig:finite_minimax_moreau}
\end{figure}
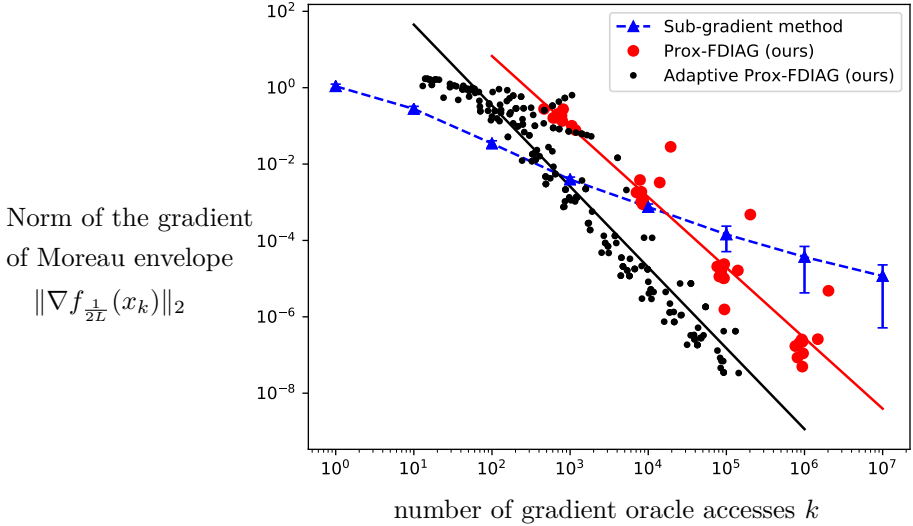

All the algorithms are initialized with the point $x_0=(4, 4)$ and are given a Lipschitzness parameter of $G = 2\,L\, \|x_0\|_2$. 
We run the algorithms ten times with randomly generated instances of the objective function $f(x)$. 
In Figure \ref{fig:finite_minimax_moreau}, we plot the norm of gradient of Moreau envelope $\| \nabla f_{\frac1{2L}}(x_k)\|_2$ against the number of iterations $k$ in log-log scale. We compute the gradient of the Moreau envelope at any point $x$, by solving the corresponding convex-concave saddle point problem \eqref{eq:moreau-reformula} using Mirror-Prox \cite{nemirovski2004prox} method with appropriate primal-dual gap based stopping criteria and then using Lemma \ref{lem:moreau-properties}(c).
For \lqafncc ({\color{red} red circles}), we show in a scatter plot the gradient norm $\| \nabla f_{\frac1{2L}}(x_{K(\varepsilon)})\|_2$ at the final output of \lqafncc $x_{K(\varepsilon)}$ versus the total number of inner iterations (of excessive gap technique) taken, for $\varepsilon=10^{0}, 10^{-1}, 10^{-2}, 10^{-3}$ over the 10 functions. For \alqafncc ({\color{black} black dots}) in a scatter plot, we plot the gradient norm $\| \nabla f_{\frac1{2L}}(x')\|_2$ at the output $x'$ of each inner sub-problem (excessive gap technique) of each inner \lqafncc step versus the total number of inner iterations (of excessive gap technique) taken to reach that point from the beginning, for $\varepsilon=10^{-7}$ over the 10 functions. For \lqafncc and \alqafncc, using solid {\color{red}red} and black (respectively) lines we also plot the best linear function (in log-scale) which fits the scatter points (using default parameters of \texttt{scipy.stats.linregress}\footnote{\url{https://docs.scipy.org/doc/scipy/reference/generated/scipy.stats.linregress.html}}). 
For the subgradient method ({\color{blue} blue triangles}), we plot the mean and standard error of gradient norm $\max_{0 \leq k' \leq k}\| \nabla f_{\frac1{2L}}(x_{\hat{k}(k')})\|_2$ over the 10 instances  at iterations $k=10^0, 10^1, \ldots, 10^7$. 
The estimate at each iteration is the best one so far in the function value, i.e.~$\hat{k}(k) \in {\arg\min}_{0 \leq k' \leq k} f(x_{k'})$. 
We see that, \lqafncc and \alqafncc have a faster convergence rate than subgradient method. Further, in the same vein as analogous variants in convex non-smooth optimization, \alqafncc is faster than \lqafncc almost always.

Subgradient method has a theoretical convergence rate of $O(\frac1{\sqrt{K}})$ for a fixed number of iterations $K$ and a constant step-size $\gamma/{\sqrt{K+1}}$   \cite[Corollary 2.2]{davis2018stochastic}. However, similar to the case of convex non-smooth problems, we observe that fixed step-size results in a  slow convergence. 
In our experiments, we achieve a faster convergence for the subgradient method by 
using a diminishing, non-summable but square-summable step-size,  $\gamma/{\sqrt{k+1}}$, which varies with the iteration number $k$. This step-size has convergence rate of $O(\frac{\log(k)}{\sqrt{k}})$ \cite[Theorem 2.1]{davis2018stochastic}, but in practice we observe a faster convergence rate than the constant step-size. After a very simple parameter search, we set $\gamma$ as $0.1 \times G \times L^{3/2}$. We ran subgradient method for a total of $K=10^7$ number of iterations. Since, subgradient method is not a descent method, at any iteration $k$, we keep track of the best point among all the points we have observed so far, $\{x_0, \cdots, x_{k-1}\}$. Ideally, we should keep track of the point with the minimum norm for the gradient of the Moreau envelope, $\|\nabla f_{\frac1{2L}}(x_k)\|_2$, but since the computation of the gradient of Moreau envelope is costly, we only keep track of the point with the minimum function value we have observed so far.
%}
\section{Conclusion}\label{sec:conc}
In this paper, we study smooth minimax problems, where the maximization is concave but the minimization is either strongly convex or nonconvex. In both of these settings, we present new algorithms improving state-of-the-art. The key ideas are i) a novel way to combine Mirror-Prox and Nesterov's AGD for strongly convex case that can tightly bound primal-dual gap and ii) an inexact prox method with good convergence rate to stationary points for the nonconvex case. While we only present our results for the Euclidean setting, generalizing it to non-Euclidean settings with the framework of Bregman divergences should be straight forward. Finally, we showcase the empirical superiority of our nonconvex algorithm over state-of-the-art subgradient method for a case of finite max-type nonconvex minimization problems. Some of the more interesting questions would be to understand the optimality of the rates that we obtain and dependence on the strong convexity parameter. Further extensions of these results to the stochastic setting would also be quite interesting.

\ifarxiv
\printbibliography
\else
\bibliographystyle{plain}
\bibliography{optimization}
\fi

\newpage
\appendix

% ========================================================================
% !TEX root = ./smooth_minimax.tex
\section*{Appendix}
% ========================================================================

\section{Nesterov's accelerated gradient descent}\label{sec:nesterov}
\begin{algorithm}[h]
	\DontPrintSemicolon % Some LaTeX compilers require you to use \dontprintsemicolon instead
	\KwIn{ Smooth concave function $h(\cdot)$, learning rate $\frac{1}{\beta}$, initial points $y_0$ and $z_0$}
	\KwOut{$y_k$}
	\For{$k = 0, 1, \ldots$} {
		$w_k \leftarrow (1-\tau_k) y_k + \tau_k z_k$, $\ y_{k+1} \leftarrow \PY{w_k + \frac{1}{\beta} \nabla h(w_k)}$, $z_{k+1} \leftarrow \PY{z_k + \eta_k \nabla h(w_k)}$
	}
	\caption{Nesterov's accelerated gradient ascent}
	\label{algo:AGD}
\end{algorithm}

Nesterov's accelerated gradient descent~\cite{nesterov1983method} is an optimal method for minimizing smooth convex functions (or equivalently maximizing smooth concave functions). In order to simplify the exposition in the sequel, we will consider the algorithm for maximizing concave functions. The pseudocode for this is presented in Algorithm~\ref{algo:AGD}. Fix any point $y \in \mathcal{Y}$. Consider the potential function
\begin{align*}
\Phi(k) \defeq k(k+1)\left(h(y)-h(y_k)\right) + 2\beta\cdot \norm{y - z_k}^2. \label{eq:agd_potential}
\end{align*}
The following lemma (from~\cite{bansal2017potential}) is the key result that helps us obtain the convergence rate of Algorithm~\ref{algo:AGD}. Here $\PY{\cdot}$ denotes projection onto $\cY$.
\begin{lemma}\label{lem:agd-pf}~\cite{bansal2017potential}
	Suppose $h(\cdot)$ is an $L$-smooth concave function and the parameters of Algorithm~\ref{algo:AGD} are chosen so that $\beta > L$, $\eta_k = \frac{k+1}{2\beta}$ and $\tau_k = \frac{2}{k+2}$. Then, we have
	\begin{align*}
	\Phi(k+1) \leq \Phi(k).
	\end{align*}
\end{lemma}
\begin{proof}[Proof of Lemma~\ref{lem:agd-pf}]
	Writing
	\begin{align}
	\Phi(k+1)-\Phi(k) = & (k+1)(k+2)\left(h(w_k)-h(y_{k+1})\right) \\ 
	&- k(k+1) \left(h(w_k)-h(y_k)\right) + 2(k+1) \left(h(y) - h(w_k)\right) \nonumber \\ &+ 2 \beta \left(\norm{z_{k+1}-y}^2 - \norm{z_{k}-y}^2\right), \label{eqn:mainterm}
	\end{align}
	we bound the three terms appearing in separate lines above. Firstly, for the third term, $\norm{z_{k+1}-y}^2 \leq \norm{z_k + \eta_k \nabla h(w_k) - y}^2 - \norm{z_{k+1}-z_k - \eta_k \nabla h(w_k)}^2$ due to Pythagoras theorem and so
	\begin{align}
	\norm{z_{k+1}-y}^2 - \norm{z_{k}-y}^2 &\leq 2 \eta_k \iprod{\nabla h(w_k)}{z_k - y} + \eta_k^2\norm{\nabla h(w_k)}^2 - \norm{z_{k+1}-z_k - \eta_k \nabla h(w_k)}^2 \nonumber \\
	&\leq 2 \eta_k \iprod{\nabla h(w_k)}{z_{k+1}-y} { - \|z_{k+1} - z_{k}\|^2}.\label{eqn:term3}
	\end{align}
	
	For the second term, we have
	\begin{align}
	&-k(k+1)\left(h(w_k)-h(y_k)\right) + 2(k+1)\left(h(y)-h(w_k)\right) \nonumber \\
	&\leq -k(k+1)\iprod{\nabla h(w_k)}{w_k - y_k} + 2(k+1) \iprod{\nabla h(w_k)}{y - w_k} = 2(k+1)\iprod{\nabla h(w_k)}{y - z_k}\label{eqn:term2}
	\end{align}
	
	Finally, for the first term, we have $h(y_{k+1})-h(w_k) \geq \iprod{\nabla h(w_k)}{{y_{k+1}}-w_k} - \frac{\beta}{2} \norm{{y_{k+1}}-w_k}^2$. Since $y_{k+1} = \argmax_{\bar{y} \in \mathcal{Y}} \iprod{\nabla h(w_k)}{\bar{y}-w_k} - \frac{\beta}{2} \norm{\bar{y}-w_k}^2$, we have for $v\defeq (1-\tau_k)y_k + \tau_k z_{k+1} \in \mathcal{Y}$,
	\begin{align}
	& h(y_{k+1})-h(w_k) \geq \iprod{\nabla h(w_k)}{{y_{k+1}}-w_k} - \frac{\beta}{2} \norm{{y_{k+1}}-w_k}^2 \nonumber \\ &\geq \iprod{\nabla h(w_k)}{{v}-w_k} - \frac{\beta}{2} \norm{{v}-w_k}^2 = \tau_k \iprod{\nabla h(w_k)}{z_{k+1}-z_k} - \frac{\beta\tau_k^2}{2} \norm{z_{k+1}-z_k}^2,\label{eqn:term1}
	\end{align}
	where we used $w_k = (1-\tau_k)y_k + \tau_k z_{k}$ in the last step.
	%	 is a smooth function. 
	Substituting~\eqref{eqn:term1},~\eqref{eqn:term2} and~\eqref{eqn:term3} in~\eqref{eqn:mainterm} proves the lemma.
\end{proof}

\section{Proofs}
\label{sec:proofs}

% ========================================================================
\subsection{Auxiliary lemma}

%As $L\|x-x_k \|^2$ is $2L$-strongly convex, the following lemma ensures strong convexity of Eq.~\eqref{eq:gradmap_general}. 
% $g(x, y) + L\|x-x_k \|^2$ is $L$-strongly convex in $x$. 
%A proof is provided in Appendix \ref{sec:weak-plus-strong-cvx_proof}.
\begin{lemma}\label{lem:weak-plus-strong-cvx}
If $f(x)$ is a $L$-weakly convex function and $\tilde{f}(x)$ is a $\tilde{\sigma}(\geq L)$-strongly convex differentiable function, then $f(x) + \tilde{f}(x)$ is $(\tilde{\sigma}-L)$-strongly convex.
\end{lemma}

\begin{proof}
%[Proof of Lemma \ref{lem:weak-plus-strong-cvx}]
\label{sec:weak-plus-strong-cvx_proof}
Since $f$ is $L$-weakly convex and $f$ is $\sigma$-strongly convex we get that,
\begin{align}
f(x') &\geq f(x) + \Ip{u_x}{x'-x} - \frac{L}{2} \|x'-x\|^2 \,, \nonumber \\
\tilde{f}(x') &\geq \tilde{f}(x) + \Ip{\nabla \tilde{f}(x)}{x'-x} + \frac{\tilde{\sigma}}{2} \|x'-x\|^2 \,, \nonumber \\
\implies f(x')+\tilde{f}(x') &\geq f(x)+\tilde{f}(x) + \Ip{u_x + \nabla \tilde{f}(x)}{x'-x} + \frac{\tilde\sigma -L}{2} \|x'-x\|^2 \,.
\end{align}
where $u_x \in \partial f(x)$. We finish the proof by noting that $\partial (f+\tilde{f}) = \partial f+ \nabla \tilde{f}$ \cite[Corollary 1.12.2.]{kruger2003frechet}.
\end{proof}

% ========================================================================
\subsection{Proof of Lemma~\ref{lem:moreau-properties}}
\label{sec:proof_moreau-properties}
We re-write $f_{\lambda}(x)$ as minimum value of a $(\frac1{\lambda} -L)$-strong convex function $\phi_{\lambda, x}$, as $f$ is $L$-weakly convex (Definition~\ref{eq:weakly-cvx}) and $\frac1{2\lambda} \|x-x'\|^2$ is differentiable and $\frac1{\lambda}$-strongly convex (Lemma \ref{lem:weak-plus-strong-cvx}),
\begin{align} \label{eq:moreau-reformula}
f_{\lambda}(x) = \min_{x' \in \cX} \bigg[ \phi_{\lambda, x}(x') = f(x') + \frac{1}{2\lambda} \|x - x'\|^2 \bigg]\,.
\end{align}
Then first part of (a) follows trivially by the strong convexity. For the second part notice the following,
\begin{align}
\min_x f_{\lambda}(x) &= \min_x \min_{x'} f(x') + \frac1{2\lambda} \|x-x' \|^2 \nonumber \\
&= \min_{x'} \min_x  f(x') + \frac1{2\lambda} \|x-x' \|^2 \nonumber \\
&= \min_{x'} f(x') \nonumber
\end{align}
Thus $\arg\min_x f_{\lambda}(x) = \arg\min_x f(x)$. For $(b)$ we can re-write the Moreau envelope $f_\lambda$ as,
\begin{align}
f_\lambda(x) &= \min_x f(x') + \frac1{2\lambda} \| x - x' \|^2 \nonumber \\
&= \frac{\|x\|^2}{2\lambda} - \frac1{\lambda} \max (x^Tx' -  \lambda f(x') - \frac{\|x'\|^2}2) \nonumber \\
&= \frac{\|x\|^2}{2\lambda} - \frac1{\lambda} \bigg(\lambda f(x') + \frac{\|x'\|^2}2\bigg)^*(x) \label{eq:moreau-conjugate}
\end{align}
where $(\cdot)^*$ is the Fenchel conjugation operator. Since $L < 1/\lambda $, using $L$-weak convexity of $f$, it is easy to see that $\lambda f(x') + \frac{\|x'\|^2}2$ is $(1 - \lambda L)$-strongly convex, therefore its Fenchel conjugate would be $\frac1{(1 - \lambda L)}$-smooth \cite[Theorem 6]{kakade2009duality}. This, along with $\frac1\lambda$-smoothness of first quadratic term \textit{}implies that $f_\lambda(x)$ is $\big(\frac1\lambda + \frac1{\lambda(1 - \lambda L)}\big)$-smooth, and thus differentiable.

For $(c)$ we again use the reformulation of $f_{\lambda}(x)$ as $\min_{x' \in \cX} \phi_{\lambda, x}(x')$ \eqref{eq:moreau-reformula}. Then by first-order necessary condition for optimality of $\hat{x}_{\lambda}(x)$, we have that $x - \hat{x}_{\lambda}(x) \in \lambda \partial f(x)$. Further, from proof of part (a) we have that $\phi_{\lambda, x}(x')$ $(1 - \lambda L)$-strongly-convex in $x'$ and it is quadratic (and thus convex) in $x$. Then we can use Danskin's theorem \cite[Section 6.11]{bertsekas2009convex} to prove that, $\nabla f_\lambda (x) = (x - \hat{x}_\lambda(x))/\lambda \in \partial f(x)$.

\subsection{Proof of Lemma~\ref{lem:weakly-cvx-smooth}}
\label{sec:proof_weakly-cvx-smooth}
It is easy to see that $g(\cdot, y)$ is $L$-weakly convex if it is $L$-smooth: $g(x', y) \geq g(x, y) + \Ip{\nabla_{x} g(x, y)}{x'-x} -  \frac{L}{2} \|x' - x\|^2$. Thus we only need to prove the case of $L$-weakly convex $g(\cdot, y)$.
Since $g(\cdot, y)$ is $L$-weakly convex we get that,
\begin{align*}
g(x', y) &\geq g(x, y) + \Ip{u_{x, y}}{x'-x} -  \frac{L}{2} \|x' - x\|^2 \nonumber \\
\implies g(x', y) + \frac{L}2 \|x'\|^2 &\geq g(x, y) + \frac{L}2 \|x\|^2 + \Ip{u_{x, y} + L x}{x'-x}
\end{align*}
where $u_{x, y} \in \partial_x g(x, y)$. This means that $\tilde{g}(x, t) \defeq g(x, y) + \frac{L}2 \|x\|^2$ is convex, since $\partial_{x} \tilde{g}(x, y) = \partial_{x} {g}(x, y) + L x$ \cite[Corollary 1.12.2.]{kruger2003frechet}.

Let $\tilde{f}(x) = \max_{y \in \cY} \tilde{g}(x, y)$. Since $\tilde{g}(x, y)$ is convex in $x$ an smooth (Definition \ref{def:smooth}), and $\cY$ is compact set we use Danskin's theorem \cite[Section 6.11]{bertsekas2009convex} to prove that,
\begin{align}
\partial \tilde{f}(x) &= {\rm conv} \{ \partial_x \tilde{g}(x, y^*) \,|\, y^* \in \arg\max_{y \in \cY} \tilde{g}(x, y)\} \,, \nonumber \\
\implies \partial {f}(x) + Lx &= {\rm conv} \{ \partial_x {g}(x, y^*) + L x \,|\, y^* \in \arg\max_{y \in \cY} {g}(x, y)\} \,, \nonumber \\
\implies \partial {f}(x) &= {\rm conv} \{ \partial_x {g}(x, y^*) \,|\, y^* \in \arg\max_{y \in \cY} {g}(x, y)\} \,. \label{eq:danskin-differential-smooth}
\end{align}
where the second to last step comes from the facts that $\partial \tilde{f} = \partial f + Lx$, $\partial_x \tilde{g}(x, y) = \partial_x g(x, y) + Lx$ \cite[Corollary 1.12.2.]{kruger2003frechet}, and $\arg\max_{y \in \cY} \tilde{g}(x, y) = \arg\max_{y \in \cY} {g}(x, y) + \frac{L}2 \|x\|^2 = \arg\max_{y \in \cY} {g}(x, y)$. Let $u_{x, y} \in \partial_x g(x, y)$ and $y^* \arg\max_{y \in \cY} g(x, y)$then,
\begin{align*}
&f(x') \geq g(x', y^*) \overset{(a)}\geq g(x, y^*) + \Ip{u_{x, y^*}}{x'-x} -  \frac{L}{2} \|x' - x\|^2 \\
\overset{(b)}\implies &f(x') \geq f(x) + \Ip{v_x}{x'-x} -  \frac{L}{2} \|x' - x\|^2 
\end{align*}
where $(a)$ uses $L$-weak convexity of $g(\cdot, y)$, and $(b)$ uses \eqref{eq:danskin-differential-smooth} and $v_x \in \partial f(x)$.

% ========================================================================
%\subsection{Proof of Lemma~\ref{lem:weakly-cvx}}
%\label{sec:proof_weakly-cvx-finite}
%
%We can re-write $f(x) = \max_{i \in [m]} f_i(x)$ as $\max_{y \in \simplex_m} g(x, y)$ where $g(x, y) = \sum_{i \in [m]} y_i f_i(x)$, where $\simplex_m  = \{y \in [0, 1]^m | \sum_{i \in [m]} y_i = 1 \}$. Since $f$ and $L$-smooth in $x$ and continuous in $y$, and $\simplex_m$ is compact set we use Danskin's theorem \cite[Section 6.11]{bertsekas2009convex} to prove that,
%\begin{align} \label{eq:danskin-differential-finite}
%\partial f(x) = {\rm conv} \{ \nabla f_{i^*}(x) \,|\, i^* \in \arg\max_{i} f_i(x)\} \,.
%\end{align}
%Then,
%\begin{align*}
%&f(x') \geq f_i(x') \overset{(a)}\geq f_i(x) + \Ip{\nabla f_i(x)}{x'-x} -  \frac{L}{2} \|x' - x\|^2 \\
%\overset{(b)}\implies &f(x') \geq f(x) + \Ip{u_x}{x'-x} -  \frac{L}{2} \|x' - x\|^2 
%\end{align*}
%where $(a)$ uses $L$-smoothness of $f_i$, and $(b)$ uses \eqref{eq:danskin-differential-finite} and  and $u_x \in \partial f(x)$.

% ========================================================================
% !TEX root = ./smooth_minimax.tex

\subsection{Proof of Theorem~\ref{thm:minimax_stronglycvx_cve_agd}}
\label{sec:minimax_stronglycvx_cve_agd_proof}
A cursory glance of the \mpscc{} (Algorithm \ref{algo:minimax_stronglycvx_cve_agd}) reveals that it is a modified version of projected accelerated gradient ascent (Algorithm~\ref{algo:AGD}) on some function of $y$ with a modified step given by \mpstep{}, which is inspired from the conceptual Mirror-Prox method of \cite{nemirovski2004prox}. In the following lemma we analyze the \mpstep{} sub-routine, which is the most non-trivial step of the algorithm.
\begin{lemma}\label{lem:mp-step}
If $\beta = 2{\frac{L^2}{\sigma}}$, the sub-routine {\em \mpstep{}($g$, $L$, $\sigma$, $w$, $\beta$, $\varepsilon_{\rm step}$)} of Algorithm \ref{algo:minimax_stronglycvx_cve_agd}, returns a pair of points $(\hat{x}_{R}, y_{R+1}) \in \cX \times \cY$, such that,
\begin{align} \label{eq:mp-step-conditions}
g(&\hat{x}_{R}, y_{R+1}) \leq \min_x g({x}, y_{R}) + \varepsilon_{\rm step} \text{,\;\;and,\;\;} 
%\\
%&
y_{R} = \PY{w + \frac1\beta \nabla_y g(\hat{x}_{R-1}, w)}
\end{align}
in $R=\ceil{\log_2 \big(({5L D_\cY}/{\sigma}) \sqrt{{L}/{2 \varepsilon_{\rm step}}}\big) }$ iterations with 
$
%\max_{r \in [R]} \ceil{ \sqrt{\frac{L}\sigma}\log \frac{(L + \sigma) \|x_0 - x^*(y_r) \|^2}{2\varepsilon_{\rm step}}} =
O\Big(\sqrt{{L}/\sigma} \log \Big(1/{\varepsilon_{\rm step}}\Big)\Big)$ gradient computations per iterations.
\end{lemma}
A proof for this lemma is provided in Appendix \ref{sec:mp-step_proof}. The above lemma guarantees that the \mpstep{} sub-routine converges fast (linear time), in $O(\log (1/\varepsilon_{\rm step}))$ steps with $O(\sqrt{{L}/\sigma}  \log^2 (1/\varepsilon_{\rm step}))$ number of gradient computations.

In the rest of the proof we will utilize the recently proposed {\em potential-function} based proof for accelerated gradient decent (AGD) \cite[Section 5.2]{bansal2017potential}. Analyzing AGD using potential-function has an advantage over the standard analysis because, even though AGD does not decrease the function value monotonically the former constructs a potential-function which monotonically decreases over the iterations. Given the guarantees (Lemma \ref{lem:mp-step}) for the \mpstep{} sub-routine we can re-write an iteration of the \mpscc{} algorithm by the following steps:
\begin{empheq}[box=\widefbox]{align}
\tau_k &= \frac2{(k+2)},\;\; \eta_k = \frac{(k+1)}{2\beta} \\
w_{k} &= (1-\tau_{k}) y_k + \tau_k z_k \\
y_{k+1} &= \PY{w_k + \frac1\beta \nabla_y h_{x_{k+1}}(w_k)} \label{eq:update_y}\\
z_{k+1} &= \PY{z_k + \eta_k \nabla_y h_{x_{k+1}}( w_k)}
\end{empheq}
where $h_{{k+1}}(y) \defeq g(x_{k+1}, y)$ such that $g(x_{k+1}, y_{k+1}) \leq  \min_{x \in \cX} g(x, y_{k+1}) + {\varepsilon_{\rm step}}$. That is at iteration $k$, \mpscc{} executes the $k$-th step of the accelerated gradient ascent for the concave function $h_{{k+1}} = g(x_{k+1}, \cdot)$ (Algorithm~\ref{algo:AGD}). As in 
%\cite[Eqn. 5.50]{bansal2017potential} 
\eqref{eq:agd_potential},
for the concave function ${h}_k:  \cY \to \reals$ and an arbitrary reference point $\ty \in \cY$, we define the following potential function for iteration $j$,
\begin{align}
\Phi^{{h}_k}(j) = j(j+1) ({h_k}(\ty) - {h_k}(y_j)) + 2\beta \|z_j - \ty\|^2
\end{align}
Since $g(x, \cdot)$ is $L$-smooth, it is also $\frac{2L^3}\sigma$-smooth ($\sigma \leq L$). 
%Then, using Eqn. (5.52) of \cite{bansal2017potential}
Then, using Lemma \ref{lem:agd-pf}
, we see that for a step-size of $\frac1\beta = \frac\sigma{2L^2}$, the potential function $\Phi^{h_k}(k)$ decrease at step of $k$ of the algorithm: $\Phi^{h_{k+1}}({k+1}) \leq \Phi^{h_{k+1}}({k})$. Thus,
\begin{align}
\Phi^{h_{k+1}}({k+1}) \leq\; &\Phi^{h_{k+1}}({k}) \nonumber \\
=\; &k(k+1) ({h_{k+1}}(\ty) - {h_{k+1}}(y_k)) + 2\beta \|z_k - \ty\|^2 \nonumber \\
=\; &k(k+1) ({h_{k}}(\ty) - {h_{k}}(y_k)) + 2\beta \|z_k - \ty\|^2 + \nonumber \\
&k(k+1) ({h_{k+1}}(\ty) -{h_{k}}(\ty)) + k(k+1) ({h_{k}}(y_k) -  {h_{k+1}}(y_k)) \nonumber \\
=\; &\Phi^{h_{k}}({k}) + k(k+1) (g(x_{k+1}, \ty) - g(x_{k}, \ty)) + k(k+1) (g(x_{k}, y_k) - g(x_{k+1}, y_k)) \nonumber \\
\overset{(a)}\leq\; &\Phi^{h_{k}}({k}) + k(k+1) (g(x_{k+1}, \ty) - g(x_{k}, \ty)) + k(k+1) \varepsilon^{(k)}_{\rm step} \label{eq:mpscc-pot-func-recurrence} \\
\overset{(b)}\implies \Phi^{h_{K}}({K}) \leq\; &\Phi^{h_{0}}({0}) + \sum_{k=0}^{K-1} k(k+1) (g(x_{k+1}, \ty) - g(x_{k}, \ty)) +  \sum_{k=1}^{K-1} k(k+1) \varepsilon^{(k)}_{\rm step} \nonumber \\
\leq\; &\Phi^{h_{0}}({0}) + (K-1)K g(x_{K}, \ty) - \sum_{k=1}^{K-1} 2k \, g(x_{k}, \ty) + \sum_{k=1}^{K-1} k(k+1) \varepsilon^{(k)}_{\rm step} \label{eq:mpscc-pot-func-telescope}
\end{align}
Where $(a)$ follows from Lemma \ref{lem:mp-step} and $g(x_{k}, y_k) - g(x_{k+1}, y_k) \leq g(x_{k}, y_k) - \min_x g(x, y_k) \leq \varepsilon^{(k)}_{\rm step}$, $(b)$ is obtained summing \eqref{eq:mpscc-pot-func-recurrence} over $k = \{ 0, \ldots, K-1 \}$. Rearranging the terms of \eqref{eq:mpscc-pot-func-telescope} we get,
\begin{align}
\Phi^{h_{0}}({0}) + \sum_{k=1}^{K-1} k(k+1) \varepsilon^{(k)}_{\rm step} &\geq \sum_{k=1}^{K-1} 2k \, g(x_{k}, \ty) + \Phi^{h_{K}}({K}) -(K-1)K g(x_{K}, \ty) \nonumber \\
&\geq  \sum_{k=1}^{K-1} 2k \, g(x_{k}, \ty) + K(K+1) (g(x_{K}, \ty) - g(x_{K}, y_K)) +\nonumber \\
&\;\;\;\;\; 2\beta\|z_K - \ty \|^2  - (K-1)K g(x_{K}, \ty) \nonumber \\
&\geq  \sum_{k=1}^{K} 2k\, g(x_{K}, \ty) - K(K+1) g(x_{K}, y_K) \nonumber  \\
&\overset{(a)}\geq  K(K+1) [ g(\bar{x}_{K}, \ty) - g(x_{K}, y_K)]  \nonumber  \\
&\overset{(b)}\geq  K(K+1) [ g(\bar{x}_{K}, \ty) - g(\tx, y_K) - {\varepsilon^{(K)}_{\rm step}}]
\end{align}
where $(a)$ uses the $\bar{x}_{K} = \frac1{K(K+1)} \sum_{k=1}^{K} (2i)\, x_i$ and convexity of $g(\cdot, \ty)$, and $(b)$ uses Lemma $6$. Thus we get that,
\begin{align}
%\max_{y \in \cY} g(\bar{x}_{K}, y) - g(\tx, y_K) &\leq  
g(\bar{x}_{K}, \ty) - g(\tx, y_K)  
%\nonumber  \\
&\leq \frac{\Phi^{h_{0}}({0})}{K(K+1)} + \sum_{k=1}^{K} \frac{k(k+1)}{K(K+1)} \varepsilon^{(k)}_{\rm step}  \nonumber  \\
&= \frac{2\beta \|y_0 - \ty \|^2}{K(K+1)} + \sum_{k=1}^{K} \frac{k(k+1)}{K(K+1)} \varepsilon^{(k)}_{\rm step}
\end{align}
Finally we get the desired general statement by taking minimum and maximum over $\tx$ and $\ty$ respectively.
%By selecting $\varepsilon^{(k)}_{\rm step} = \frac1{k^2(k+1)}$ we get,
%\begin{align}
%\max_{\ty \in \cY} g(\bar{x}_{K}, \ty) -  \min_{\tx \in \cX} g(\tx, y_K)
%&\leq \frac{4\frac{L^2}{\sigma} D_{\cY}^2 + \ln K + 1}{K(K+1)}
%\end{align}
By selecting $\varepsilon^{(k)}_{\rm step} = \frac{L^2 D^2_{\cY}}{\sigma k^3(k+1)}$ we get,
\begin{align}
\max_{\ty \in \cY} g(\bar{x}_{K}, \ty) -  \min_{\tx \in \cX} g(\tx, y_K)
&\leq \frac{6\frac{L^2}{\sigma} D_{\cY}^2}{K(K+1)}
\end{align}
Further, using Lemma~\ref{lem:mp-step} and $\varepsilon^{(k)}_{\rm step} = \frac{L^2 D^2_{\cY}}{\sigma k^3(k+1)}$, we get that the total number of gradient computations at iteration $k$ is at most $O\big(\sqrt\frac{L}{\sigma} \log^2 (k)\big)$:
\begin{align}
%\max_{r \in [R]} 
\ceil[\bigg]{\log_2 5 k^2 \sqrt{\frac{L}\sigma}} \, 
%\ceil{ \sqrt{\frac{L}\sigma}\log \frac{k^4 \sigma (L + \sigma) \|x_0 - x^*(y_r) \|^2}{L^2 D^2_{\cY}}}
O\Big(\sqrt{\frac{L}\sigma} \log \big(k^4 \big)\Big)
\end{align}

Note that in updating $y_{k+1}$ in Eq.~\eqref{eq:update_y} and $x_{k+1}$ in \mpstep{} sub-routine, 
we were applying the principle of conceptual Mirror-Prox, where the update needs to satisfy some fixed point equation. 
This is critical in proving the above fast convergence rate. 
%A  standard AGD update is $y_{k+1} = w_k + (1/\beta) \nabla_y h_{x_k}(w_k)$ with $x_k$ satisfying $g(x_{k},y_{k-1}) \leq \min_x g(x,y_{k) + \varepsilon_{\rm step}$. 

% --------------------------------------------------------------------------------------------------------------------------------------------------------------------------------
\subsubsection{Proof of Lemma~\ref{lem:mp-step}} \label{sec:mp-step_proof}
For brevity, we define the following operations,
%$x^*(y) = \underset{x \in \cX}{\arg\min} \; g(x, y)$,  $\hat{x}(y) = \underset{x \in \cX}{\arg\min} \; \max\{g(x, y) -  g(x^*(y), y), \varepsilon\}$, $y^+ = w + \nabla_y g(x^*(y), w)/\beta$, and $\hat{y}^+ = w + \nabla_y g(\hat{x}(y), w)/\beta$. 
\begin{align}
x^*(y) &= \underset{x \in \cX}{\arg\min} \; g(x, y) \label{eq:op-x-minimizer} \\
%\hat{x}(y) &\in \underset{x \in \cX}{\arg\min} \; \max\{g(x, y) -  g(x^*(y), y), \varepsilon\} \\
y^+ &= \PY{w + \frac1\beta \nabla_y g(x^*(y), w)} \label{eq:op-y-ascent} 
%\\ \hat{y}^+ &= w + \frac1\beta \nabla_y g(\hat{x}(y), w)\, .
\end{align}
$x^*(y)$ is unique since $g(\cdot, y)$ is strongly convex. We first prove that, $x^*(y)$ is ${\frac{L}\sigma}$-Lipschitz continuous as follows.
%We use optimality of $x^*(y_1)$ \eqref{eq:op-x-minimizer}, $\sigma$-strong convexity of $g(\cdot, y)$ and concavity of $g(x, \cdot)$ to get,
%\begin{align}
%\frac{\sigma}2 \| x^*(y_1) -x^*(y_2) \|^2 + g(x^*(y_1), y_1) &\leq g(x^*(y_2), y_1) \nonumber \\
%&\leq g(x^*(y_2), y_2) + \Ip{\nabla_y g(x^*(y_2), y_2)}{y_1 - y_2}  \label{eq:mp-lipschitz-1}\,.
%\end{align}
%Now by interchanging the indices $1, 2$ of \eqref{eq:mp-lipschitz-1} and adding it to \eqref{eq:mp-lipschitz-1} we get,
%\begin{align}
%{\sigma} \| x^*(y_1) -x^*(y_2) \|^2  &\leq -\Ip{\nabla_y g(x^*(y_1), y_1) - \nabla_y g(x^*(y_2), y_2)}{y_1 - y_2} \nonumber \\
%&= -\Ip{\nabla_y h(y_1) - \nabla_y h(y_2)}{y_1 - y_2} + \nonumber \\
%&\leq {G} \| y_1 - y_2 \|^2 \label{eq:mp-lipschitz-2}
%\end{align}
%by using the fact that $h(y) =  \max_x g(x, y)$ is concave. 
\begin{align}
\sigma \|x^*(y_2) - x^*(y_1) \|^2 &\overset{(a)}\leq 
\Ip{\nabla_x g(x^*(y_2), y_2) - \nabla_x g(x^*(y_1), y_2)}{x^*(y_2) - x^*(y_1)} \nonumber \\
&\overset{(b)}\leq \Ip{ - \nabla_x g(x^*(y_1), y_2)}{x^*(y_2) -x^*( y_1) } \nonumber \\
&\overset{(c)}\leq \Ip{\nabla_x g(x^*(y_1), y_1) - \nabla_x g(x^*(y_1), y_2)}{x^*(y_2) - x^*(y_1)} \nonumber \\
&\overset{(d)}\leq  L \| y_1 - y_2 \| \| x^*(y_2) - x^*(y_1) \| \label{eq:mp-lipschitz-2}
\end{align}
where $(a)$ uses $\sigma$-strong convexity of $g(\cdot, y)$, $(b)$ and $(c)$ use the necessary first order optimality conditions for $x^*(y_1)$ and $x^*(y_2)$: $\Ip{\nabla_x g(x^*(y), y)}{x - x^*(y)} \geq 0$, and $(d)$ uses Cauchy-Schwarz inequality and $L$-smoothness of $g$ (Definition \ref{def:smooth}). Next we prove that the operation $(\cdot)^+$ is a contraction as follows,
\begin{align}
\| y_1^+ - y_2^+ \| &= \| \PY{w + \frac1\beta \nabla_y g(x^*(y_1), w)}  - \PY{w + \frac1\beta \nabla_y g(x^*(y_2), w)}  \| \nonumber \\
&\overset{(a)}\leq \frac1\beta  \| \nabla_y g(x^*(y_1), w)  - \nabla_y g(x^*(y_2), w)  \| \nonumber \\
&\overset{(b)}\leq \frac{L}{\beta} \| x^*(y_1) -x^*(y_2) \| \nonumber \\
&\overset{(c)}\leq \frac{L}{\beta} {\frac{L}{\sigma}} \| y_1 - y_2 \| \overset{(d)}\leq 2^{-1} \| y_1 - y_2 \| \label{eq:mp-contract-1}
\end{align}
where $(a)$ uses Pythagorean theorem and \eqref{eq:op-y-ascent}, $(b)$ uses $L$-smoothness of $g$, $(c)$ uses \eqref{eq:mp-lipschitz-2}, and $(d)$ uses $\beta = 2{\frac{L^2}{\sigma}}$. Therefore as $(\cdot)^+$ is a contraction by Banach's fixed point theorem, it has a unique fixed point $\ty$: $(\ty)^+ = \ty$, as $\cY$ is a compact (and hence complete) metric space. Now we will prove that the output of \mpstep{}, $(\hat{x}_{R}, y_{R+1})$ satisfies \eqref{eq:mp-step-conditions}. Notice that if $\varepsilon_{\rm agd}$ is small then $\hat{x}_r$ is close to $x^*(y_r)$:
\begin{align}
\frac\sigma2 \| \hat{x}_{r} - {x}^*(y_{r}) \|^2 \overset{(a)}\leq g(\hat{x}_r, y_r) -  \min_x g(x, y_r) 
%\overset{(b)}\leq \varepsilon_{\rm agd} 
\overset{(b)}\implies  
\| \hat{x}_{r} - {x}^*(y_{r}) \| \leq \sqrt{\frac{2\varepsilon_{\rm agd}}{\sigma}} = \frac{\beta \varepsilon_{\rm mp}}{4L} \label{eq:mp-step-approx-minimzer}
\end{align}
where $(a)$ uses $\sigma$-strong convexity and optimality of $x^*(y_r)$, and $(b)$ uses \eqref{eq:mp-step-agd-error}, and $(c)$ uses $\varepsilon_{\rm agd} = {\sigma \beta^2 \varepsilon_{\rm mp}}/({32 L^2})$. Next we see that $\| y_r - {\ty} \| $ decreases to $\varepsilon$ exponentially fast.
\begin{align}
\| y_r - {\ty} \| &\overset{(a)}= \| \PY{w + \frac1\beta \nabla_y g(\hat{x}_{r-1}, w)} - {(\ty)}^+\| \nonumber \\
&\overset{(b)}\leq \| y_{r-1}^+ - {(\ty)}^+ \| + \| \PY{w + \frac1\beta\nabla_y g({x}^*(y_{r-1}), w)} - \PY{w + \frac1\beta \nabla_y g(\hat{x}_{r-1}, w)} \|  \nonumber \\
&\overset{(c)}\leq 2^{-1} \| y_{r-1} - {\ty} \| + \frac{L}\beta \| {x}^*(y_{r-1})  - \hat{x}_{r-1}\| \nonumber \\
%&\overset{(d)}\leq 2^{-1} \| y_{r-1} - {\ty} \| + \frac{L}\beta  \sqrt{\frac{2 \varepsilon_{\rm agd}}{\sigma}} \nonumber \\
&\overset{(d)}\leq 2^{-1} \| y_{r-1} - {\ty} \| +  \frac{\varepsilon_{\rm mp}}4 \label{eq:mp-step-reccurence} \\
&\overset{(e)}\leq 2^{-r} \| y_{0} - {\ty} \| + \frac{\varepsilon_{\rm mp}}2 \label{eq:mp-step-error-1} 
\end{align}
%where $(a)$ uses \eqref{eq:mp-step-agd-error} and the fact that $\ty=(\ty)^+$ is a fixed point, % \eqref{eq:mp-step-y-step}, 
where $(a)$ uses $y_{r+1} = \mathcal{P}_{\cY} \big(w + \frac1\beta \nabla_y g(\hat{x}_{r}, w)\big)$ and the fact that $\ty=(\ty)^+$ is a fixed point, 
$(b)$ uses triangular inequality and \eqref{eq:op-y-ascent}, $(c)$ uses \eqref{eq:mp-contract-1}, Pythagorean theorem and $L$-smoothness of $g$ (Definition \ref{def:smooth}), $(d)$ uses \eqref{eq:mp-step-approx-minimzer}, and $(e)$ just unrolls the recurrence relation in \eqref{eq:mp-step-reccurence} . 
Next, we prove that the minimizer at $y_{R+1}$, $x^*(y_{R+1})$ is not far from $\hat{x}_R$.
\begin{align}
\| {x}^*(y_{R+1}) - \hat{x}_R\|  &\overset{(a)}\leq \| {x}^*(y_{R+1}) - x^*(\ty) \| + \| x^*(\ty) - x^*(y_{R})\| + \| {x}^*(y_{R}) - \hat{x}_{R} \| \nonumber  \\
&\overset{(b)}\leq \frac{L}\sigma ( \| y_{R+1} - \ty \| + \| y_{R} - \ty \| ) +  \frac{\beta {\varepsilon_{\rm mp}}}{4L} \nonumber  \\
&\overset{(c)}\leq \frac{L}\sigma ( {\varepsilon_{\rm mp}} + {\varepsilon_{\rm mp}}) +  \frac{L {\varepsilon_{\rm mp}}}{2\sigma} = \frac{5L {\varepsilon_{\rm mp}}}{2\sigma} \label{eq:mp-approx-min-close-prev-min}
\end{align}
where $(a)$ uses triangle inequality, and $(b)$ uses  \eqref{eq:mp-lipschitz-2} and \ref{eq:mp-step-approx-minimzer}, and $(c)$ uses \eqref{eq:mp-step-error-1} and the fact that $R = \ceil{\log_2 \frac{2D_\cY}{\varepsilon_{\rm mp}}}$. Finally, we prove that $(x_R, y_{R+1})$ satisfies \eqref{eq:mp-step-conditions}.
\begin{align} 
g(\hat{x}_{R}, y_{R+1}) &\overset{(a)}\leq g(x^*(y_{R+1}),y_{R+1})+ \Ip{\nabla_x g(x^*(y_{R+1}), y_{R+1})}{\hat{x}_R - x^*(y_{R+1}),}  + \frac{L}2 \| {x}^*(y_{R+1}) - \hat{x}_R \|^2 \nonumber \\
&\overset{(b)}\leq \min_x g(x, y_{R+1})+ 0 +\frac{25 L^3 {\varepsilon^2_{\rm mp}}}{8 \sigma^2} 
\overset{(c)}= \min_x g(x, y_{R+1})+ \varepsilon_{\rm step}
\end{align}
\begin{sloppy}
where $(a)$ uses $L$-smoothness of $g(\cdot, y)$, $(b)$ uses necessary first order optimality condition: ${\Ip{\nabla_x g(x^*(y), y)}{x - x^*(y)}= 0}$ and \eqref{eq:mp-approx-min-close-prev-min}, and $(c)$ uses $\varepsilon_{\rm mp} = \frac{2 \sigma}{5 L}\sqrt{\frac{2 \varepsilon_{\rm step}}{L}}$.
\end{sloppy} 

Let the number of gradient computations done per iteration of \mpstep{} (a run of accelerated gradient ascent) be $T_r$ and $\kappa =\sqrt{L/\sigma}$. Then, from guarantee on AGD (\cite[Eqn. (5.68)]{bansal2017potential}), we get that,
\begin{align}
g(\hat{x}_r, y_r) - g(x^*(y_r), y_r) &\leq \Big(1 + \frac1{\sqrt{\kappa} - 1}\Big)^{-T_r}  {\Big(g(x_0, y_r) - g(x^*(y_r), y_r) + \frac{\sigma}2 \|x_0 - x^*(y_r)\|^2\Big)} \nonumber \\
&\leq e^{-T_r/\sqrt{\kappa}} \; 2\,(g(x_0, y_r) - g(x^*(y_r), y_r)) \nonumber \\
&\leq e^{-T_r/\sqrt{\kappa}} \; 2\,(f(x_0) - h( y_r)) \nonumber \\ %\label{eq:agd-steps-ub}
&\leq e^{-T_r/\sqrt{\kappa}} \; 2\,(f(x_0) - \min_{y' \in D_{\cY}} h(y')) \label{eq:agd-steps-ub}\,,
\end{align}
%Now notice that by compactness of $\cY$ and $L(1 + L/\sigma)$-smoothness of $h$ (Lemma \ref{lem:outer-max-smooth}),
%\begin{align}
% h(y_r) &\geq h^*+ \Ip{\nabla h(y^*)}{y_r - y^*} - \frac1{2} {L(1 + \frac{L}\sigma)} \|y_r - y^* \|^2 \nonumber \\
%&\geq f^*- \|{\nabla h(y^*)}\| \|{y_r - y^*}\| - \frac1{2} {L(1 + \frac{L}\sigma)} D_{\cY}^2 \nonumber \\
%&\geq f^*- \max_{y' \in D_{\cY}} \|{\nabla h(y')}\| D_{\cY} - \frac1{2} {L(1 + \frac{L}\sigma)} D_{\cY}^2  \label{eq:h-smooth-lb}
%\end{align}
where $\min_{y' \in D_{\cY}} h(y')$ is well-defined since $\cY$ is compact and $h$ is smooth (Lemma \ref{lem:outer-max-smooth}). This means that if we want $g(\hat{x}_r, y_r) - g(x^*(y_r), y_r)  \leq \varepsilon_{\rm agd}$, then required number of steps $T_r$ is at most,
\begin{align}
\ceil[\bigg]{\sqrt{\frac{L}\sigma}\log \frac{2(f(x_0) - \min_{y' \in D_{\cY}} h(y'))}{\varepsilon_{\rm agd}}} &= \ceil[\bigg]{\sqrt{\frac{L}\sigma}\log \frac{50 L (f(x_0) - \min_{y' \in D_{\cY}} h(y'))}{\sigma \varepsilon_{\rm step}}} \nonumber \\ &=O\Big(\sqrt{\frac{L}\sigma} \log \Big(\frac1{\varepsilon_{\rm step}}\Big)\Big)
\end{align}

\subsection{Proof of Lemma \ref{lem:outer-max-smooth}}
\label{sec:outer-max-smooth_proof}
We know that $h(y) = \min_{x \in \cX} g(x, y)$, where $g(\cdot, y)$ is $\sigma$-strongly convex, $g(x, \cdot)$ is concave,  $g$ is $L$-smooth (Definition \ref{def:smooth}). Since $g(\cdot, y)$ is strongly convex, the minimizer $x^*(y) = \arg\min_{x \in \cX} g(x, y)$ unique. Then by Danskin's theorem \cite[Section 6.11]{bertsekas2009convex}, $h$ is differentiable and $\nabla h(y) = \nabla_y g(x^*(y), y)$. Then $h$ can be show to be smooth as follows,
\begin{align}
\|\nabla h(y_1) - \nabla h(y_1) \| &= \| \nabla_y g(x^*(y_1), y_1) - \nabla_y g(x^*(y_2), y_2) \| \nonumber \\
&\leq \| \nabla_y g(x^*(y_1), y_1) - \nabla_y g(x^*(y_1), y_2) \| + \| \nabla_y g(x^*(y_1), y_2) - \nabla_y g(x^*(y_2), y_2) \| \nonumber \\
&\overset{(a)}\leq L \| y_1 - y_2 \| + L \| x^*(y_1) - x^*(y_2) \| \nonumber \\
&\overset{(b)}\leq L \| y_1 - y_2 \| + L \frac{L}\sigma \|y_1 - y_2 \| = L \big(1 + \frac{L}\sigma) \|y_1 - y_2 \| 
\end{align}
where $(a)$ uses $L$-smoothness of $g$ and $(b)$ uses \eqref{eq:mp-lipschitz-2}.

\subsection{Proof of Corollary~\ref{thm:finite_minimax_gd}}
\label{sec:proof_finite_minimax_gd}

\begin{algorithm}[]
	\DontPrintSemicolon % Some LaTeX compilers require you to use \dontprintsemicolon instead
	\KwIn{functional components $\{f_i\}_{i=1}^m$, Lipschitzness $G$, smoothness $L$, domain $\cX$, target accuracy $\varepsilon$, initial point $x_0$}
	\KwOut{$x_k$}
	$\tepsilon \leftarrow \frac{\varepsilon^2}{64\,L}$\;
	\For{$k = 0, 1, \ldots$} {
		%$x_{k+1} \gets \min_x \max_i f_i (x_k) + \Ip{\nabla f_i(x_k)}{x - x_k} + L/2 \| x - x_k \|^2$ \Comment*[r]{excessive gap \cite{nesterov2005excessive}}  
		Using excessive gap technique \cite[Problem (7.11)]{nesterov2005excessive} for strongly convex components, 
		find $x_{k+1} \in \cX$ such that,
		\begin{align}
		\label{eq:finite_inner}
		& \wf(x_{k+1};x_k)  \;\; \leq \;\; \min_x \wf(x;x_k) + \tepsilon/4
		%\max_i f_i (x_k) + \Ip{\nabla f_i(x_k)}{x_{k+1} - x_k} + G/2 \| x_{k+1} - x_k \|^2 \nonumber \\
		%		\leq &\min_{x \in \cX} \max_i f_i (x_k) + \Ip{\nabla f_i(x_k)}{x - x_k} + G/2 \| x - x_k \|^2 + \tepsilon/4
		\end{align}
		\If{$f(x_k) - 3\tepsilon/4 < \wf(x_{k+1};x_k)$} {%\max_i f_i (x_k) + \Ip{\nabla f_i(x_k)}{x_{k+1}- x_k} + G/2 \| x_{k+1} - x_k \|^2$}{
			\KwRet{$x_{k}$}
			%			\KwRet{$\hat{x} \in {\arg\min}_{x \in \cX} f(x) + G \|x -x_k\|^2$}\\
			%			Using Nesterov's scheme for minimization of smooth-strongly-convex minima problem \cite[Scheme 2.3.13]{nesterov1998introductory},  find, $\hat{x} \in \cX$ such that,
			%			\begin{align}
			%			f(\hat{x}) + G \|\hat{x} -x_k\|^2 \leq {\min}_{x \in \cX} f(x) + G \|x -x_k\|^2 + \tilde{\tepsilon}
			%			\end{align}
		} 
	}
	\caption{Proximal Finite Dual Implicit Accelerated Gradient (\lqafncc) for finite nonconvex concave minimax optimization}
	\label{algo:minimax_gd}
\end{algorithm}

Let 
\begin{align}
\wf(x; x_k) \;\; = \;\; \max_{1 \leq i \leq m} \; f_i (x_k) + \Ip{\nabla f_i(x_k)}{x - x_k} + \frac{L}{2} \| x - x_k \|^2 
\label{eq:quad-approx-finite}
\end{align}
be a quadratic approximation of the finite max-type function $f(x)$ at $x_k$. Then, $\wf(\cdot; x_k)$ is $L$-strongly convex, since it is a maximum of convex functions and the quadratic term in \eqref{eq:quad-approx-finite} is independent of $i$.

Proof is similar to that of Theorem \ref{thm:minimax_noncvx_cve_gd}. We divide the analysis of each iteration of our algorithm into two cases. 

\bigskip
\noindent
{\bf Case 1: 
	%{\bf If $\mathbf{f(x_{k+1}; x_k) {\boldsymbol\leq}  f(x_k) - 3\boldsymbol{\tepsilon}/4}$},  
	${\wf(x_{k+1}; x_k) {\leq}  f(x_k) - 3{\tepsilon}/4}$.}  
At iteration $k$ the objective value decreases by at least $3{\tepsilon}/4$. 
One cannot have more than 
$\ceil[\Big]{\frac{4(f(x_0) - f^*)}{3\tepsilon}}$ Case 1 steps, before termination. 

\bigskip
\noindent 
{\bf Case 2: 
	%{\bf If $\mathbf{f(x_{k+1}; x_k) \boldsymbol>  f(x_k) - 3\boldsymbol{\tepsilon}/4}$}
	${\wf(x_{k+1}; x_k) {>}  f(x_k) - 3{\tepsilon}/4}$:} 
We show that $x_k$ is an $\varepsilon$-FOSP as follows. 
\begin{align}
f(x_k) - \frac{3\tepsilon}4 < \wf(x_{k+1}; x_k) \leq \min_x \wf(x; x_k)  + \frac{\tepsilon}4  
\;\;\;\; \implies   \;\;\;\; f(x_k)  < \min_x \wf(x; x_k)  + {\tepsilon}
\label{eq:less-decrease-finite}
\end{align}
Define $x_k^*$ as the point satisfying  $x^*_k = \arg\min_x \wf(x; x_k)$. 
By $L$-strong convexity of $\wf(\cdot, x_k)$ \eqref{eq:quad-approx-finite}, we prove that $x_k$ is close to $x_k^*$:
\begin{align}
&\wf(x^*_k;x_k) + \frac{L}2 \|x_k-x^*_k\|^2 \; \leq \; \wf(x_{k};x_k) \; =\;  f(x_k) \; \overset{(a)}<  \;\wf(x^*_{k};x_k) + \tepsilon \nonumber \\
\implies \;\; &\|x_k-x^*_k\| < \sqrt{\frac{2\tepsilon}{L}}  \label{eq:moreau-ball-ub-finite}
\end{align}
where $(a)$ uses \eqref{eq:less-decrease-finite}. Now consider any $\tx \in \cX$, such that $4\sqrt{\tepsilon/L} \leq \| \tx - x_k\|$. Then,
\begin{align}
f(\tx) + L \|\tx - x_k\|^2 &= \max_i f_i (\tx) +  L \|\tx - x_k\|^2 \nonumber \\
&\overset{(a)}\geq \max_i f_i (\tx) +  \Ip{\nabla f_i(x_k)}{\tx - x_k} + \frac{L}2 \|\tx - x_k\|^2 \nonumber \\
&\overset{(b)}= \wf(\tx; x_k) \nonumber \\
&\overset{(c)}\geq \wf(x_k^*; x_k) + \frac{L}2 \|\tx - x^*_k\|^2 \nonumber \\
&\overset{(d)}\geq f(x_k) - \tepsilon + \frac{L}2 (\|\tx - x_k\| - \|x_k - x^*_k\|)^2 \nonumber \\
&\overset{(e)}\geq f(x_k) - \tepsilon + 2 \tepsilon 
%\nonumber \\ &
\;\; \overset{}=  \;\; f(x_k) + \tepsilon \label{eq:moreau-ball-lb-finite}
\end{align}
where $(a)$ uses weak convexity of $f_i$, $(b)$ uses \eqref{eq:quad-approx-finite}, $(c)$ uses $L$-strong convexity of $\wf(\cdot; x_k)$ at its minimizer $x^*_k$, $(d)$ uses \eqref{eq:less-decrease-finite}, and $(b)$ and $(e)$ use triangle inequality, \eqref{eq:moreau-ball-ub-finite} and $4\sqrt{\tepsilon/L} \leq \| \tx - x_k\|$.

Now consider the Moreau envelope, $f_{\frac1{2L}}(x) = \min_{x' \in X} \phi_{\frac1{2L}, x}(x')$ where $\phi_{\lambda,x}(x') = f(x') + L \|x - x'\|^2$. Then, we can see that $\phi_{\frac1{2L}, x_k}(x')$ achieves its minimum in the ball $\{x' \in \cX \,|\, \|x' - x_k\| \leq 4\sqrt{\tepsilon/L} \}$ by \eqref{eq:moreau-ball-lb-finite} and Lemma \ref{lem:moreau-properties}(a). Thus, with Lemma \ref{lem:moreau-properties}(b,c), we get that,
\begin{align}
\| \nabla f_{\frac1{2L}}(x_k) \| \leq (2L) \| x_k - \hat{x}_{1/2L}(x_k) \| = 8 \sqrt{L \tepsilon} = \varepsilon
\end{align}

% Excessive gap technique removed
Now we use the excessive gap technique for non-smooth strongly convex functions with max-structure to solve the inner optimization problem in $4 G (m \log m) \sqrt{\frac{ \log m}{\tepsilon L}}$ computations \cite[Problem (7.11)]{nesterov2005excessive}.

Putting these together we see that the total number of inner steps to reach $\varepsilon$-FOSP is,
\begin{align}
\ceil[\bigg]{\frac{4 (f(x_0) - f^*)}{3 \tepsilon}} \ceil[\bigg]{2G (m \log m) \sqrt{\frac{ \log m}{L \tepsilon}}} = \ceil[\bigg]{\frac{4^4 L (f(x_0) - f^*)}{3 \varepsilon^2}} \ceil[\bigg]{\frac{2^5 G}{\varepsilon} (m \log^{3/2} m) {{}{}}}
\end{align}

\if0
{\color{cyan}
% Using strongly convex concave algorithm instead
Now for solving the inner optimization problem, using Algorithm~\ref{algo:minimax_stronglycvx_cve_agd}, we solve the following equivalent $L$-smooth $L$-strongly-convex($x$)--concave($y$) min-max problem,
\begin{align}
\min_{x \in \cX} \bigg[ \max_{y \in \simplex_m} \bigg[ g(x, y) = \sum_{i \in [m]} y_i f_i(x, x_k) =  \sum_{i \in [m]} y_i \; \Big[ f_i (x_k) + \Ip{\nabla f_i(x_k)}{x - x_k} + \frac{L}{2} \| x - x_k \|^2 \Big] \bigg] \bigg]
\end{align}
such that $\simplex_m = \{y \in [0, 1]^m \,|\, \sum_i y_i = 1 \}$ up to a primal-dual gap of $\frac\tepsilon{4}$, which implies the same amount of optimality gap in $f(\cdot, x_k)$. Using Theorem \ref{thm:minimax_stronglycvx_cve_agd}, we see that this requires $O(\frac{L}\varepsilon \log^2 (\frac{1}\varepsilon))$ computations per inner optimization step of Algorithm \ref{algo:minimax_gd}.

the inner optimization problem in $4 G (m \log m) \sqrt{\frac{ \log m}{\tepsilon L}}$ computations \cite[Problem (7.11)]{nesterov2005excessive}.

Putting these together we see that the total number of inner steps to reach $\varepsilon$-FOSP is,
\begin{align}
\ceil[\bigg]{\frac{4 (f(x_0) - f^*)}{3 \tepsilon}} \ceil[\bigg]{2G (m \log m) \sqrt{\frac{ \log m}{L \tepsilon}}} = \ceil[\bigg]{\frac{4^4 L (f(x_0) - f^*)}{3 \varepsilon^2}} \ceil[\bigg]{\frac{2^5 G}{\varepsilon} (m \log^{3/2} m) {{}{}}}
\end{align}
}
\fi

%{\color{blue}
\subsection{\alqafncc algorithm}
\label{sec:adaptive_finite_minimax_gd}
In this section, we provide the \alqafncc (Algorithm~\ref{algo:minimax_gd_adapt}) to find an $\varepsilon$-FOSP of the finite max-type nonconvex minimax problem \ref{prob:finite-minimax} with $L$-smooth components. \alqafncc is a variation of the \lqafncc (Algorithm \ref{algo:minimax_gd}). \alqafncc uses \lqafncc as a sub-routine and successively finds $\varepsilon'$-FOSPs, for geometrically decreasing values of $\varepsilon'$ starting from $\varepsilon_0$ ($\geq \varepsilon$) until $\varepsilon'$ becomes equal to $\varepsilon$. It uses the $\varepsilon'$-FOSP as the starting point to find an $\varepsilon'/2$-FOSP. In the following corollary, we show that \alqafncc has the same the first-order oracle complexity (up to a $O(\log(\frac1{\varepsilon}))$ factor) as the \lqafncc.
\begin{corollary}[Convergence rate of \alqafncc]
	If the functional components $f_i(x)$'s are $G$-Lipschitz and $L$-smooth,  and the optimal solution is 
	bounded below, i.e.~$f(x)\geq f^* > -\infty$, then after: $K=\ceil[\bigg]{\log_2 \frac{\varepsilon_0}{\varepsilon}}$ {\em outer} steps,  \alqafncc outputs an $\varepsilon$-FOSP. The total first-order oracle complexity to find $\varepsilon$-FOSP is:  $\ceil[\bigg]{\log_2 \frac{\varepsilon_0}{\varepsilon}} \ceil[\bigg]{\frac{4^4 L (f(x_0) - f^*)}{3 \varepsilon^2}} \cdot \ceil[\bigg]{\frac{2^4 G}{\varepsilon} (m \log^{3/2} m) {{}{}}}$. %	total iterations by the inner excessive gap technique step.
	\label{thm:adap_finite_minimax_gd}
\end{corollary}
\begin{proof}
Notice that, each iteration of \alqafncc for finding an $\varepsilon'$-FOSP, is a run of \lqafncc (Algorithm \ref{algo:minimax_gd}), which has a maximum first-order oracle complexity of $\ceil[\bigg]{\frac{4^4 L (f(x_0) - f^*)}{3 \varepsilon^2}} \cdot \ceil[\bigg]{\frac{2^4 G}{\varepsilon} (m \log^{3/2} m) {{}{}}}$ for finding an $\varepsilon'$-FOSP (Corollary \ref{thm:finite_minimax_gd}), as $\varepsilon \leq \varepsilon'$. Further, since $\varepsilon'$ starts at $\varepsilon_0$ and halves after each iteration until $\varepsilon'$ becomes less than or equal to $\varepsilon$, the total number of outer iterations is $K = \ceil[\bigg]{\log_2 \frac{\varepsilon_0}{\varepsilon}}$.
\end{proof}
\noindent
Therefore, \alqafncc has the same first-order oracle complexity as \lqafncc, up to a $O(\log(\frac1{\varepsilon}))$ factor. However, we observe that \alqafncc converges faster than \lqafncc in our experiments.

\begin{algorithm}[]
	\DontPrintSemicolon % Some LaTeX compilers require you to use \dontprintsemicolon instead
	\KwIn{functional components $\{f_i\}_{i=1}^m$, Lipschitzness $G$, smoothness $L$, domain $\cX$, target accuracy $\varepsilon$, initial point $x_0$, initial accuracy $\varepsilon_0$}
	\KwOut{$x_k$}
	$\varepsilon' \leftarrow \max(\varepsilon_0,\;\varepsilon)$\;
	\For{$k = 0, 1, \ldots$}{
	Using Prox-FDIAG (Algorithm \ref{algo:minimax_gd}) initialized at $x_k$, find $x_{k+1} \in \cX$ such that $x_{k+1}$ is an $\varepsilon'$-FOSP (Definition \ref{def:eps_fosp}) of the function $f(x) = \max_{1 \leq i \leq m} f_i(x)$\;
	\If{$\varepsilon = \varepsilon'$} {
	$k \leftarrow k+1$\;
	\KwRet{$x_{k}$}
	}
	\Else{$\varepsilon' \leftarrow \max(\frac{\varepsilon'}{2},\;\varepsilon)$}
	}
	\caption{Adaptive Proximal Finite Dual Implicit Accelerated Gradient (\alqafncc) for finite nonconvex concave minimax optimization}
	\label{algo:minimax_gd_adapt}
\end{algorithm}
%}

\end{document}